\documentclass[a4paper,10pt]{article}
\usepackage{test,paralist,graphicx,subcaption,booktabs,caption,multirow}
\usepackage[authoryear,round]{natbib}
\usepackage{hyperref}
\usepackage{doi}
\usepackage{authblk}

\title{Efficient Minimum Distance Estimation of Pareto Exponent from Top Income Shares}
\author[1]{Alexis Akira Toda\thanks{Email: \href{mailto:atoda@ucsd.edu}{atoda@ucsd.edu}.}}
\affil[1]{Department of Economics, University of California San Diego}
\author[2]{Yulong Wang\thanks{Email: \href{mailto:ywang402@maxwell.syr.edu}{ywang402@maxwell.syr.edu}.}}
\affil[2]{Department of Economics, Maxwell School, Syracuse University}

\numberwithin{equation}{section}
\numberwithin{thm}{section}

\begin{document}
\maketitle

\begin{abstract}

We propose an efficient estimation method for the income Pareto exponent when only certain top income shares are observable. Our estimator is based on the asymptotic theory of weighted sums of order statistics and the efficient minimum distance estimator. Simulations show that our estimator has excellent finite sample properties. We apply our estimation method to U.S.\ top income share data and find that the Pareto exponent has been stable at around 1.5 since 1985, suggesting that the rise in inequality during the last three decades is mainly driven by redistribution between the rich and poor, not among the rich.

\medskip


\medskip

\end{abstract}

\section{Introduction}

It is well-known that the income distribution as well as many other size distributions of economic interest exhibit Pareto (power law) tails,\footnote{\cite{Pareto1896LaCourbe,Pareto1897Cours} discovered that the rank size distribution of income shows a straight line pattern on a log-log plot, which implies a power law. The power law in size distributions of economic variables has been documented for city size \citep{Auerbach1913,zipf1949,gabaix1999,giesen-zimmermann-suedekum2010,RozenfeldRybskiGabaixMakse2011}, firm size \citep{Axtell2001}, wealth \citep{KlassBihamLevyMalcaiSolomon2006,Vermeulen2018}, and consumption \citep{TodaWalsh2015JPE,Toda2017MD}, among others. See \cite{gabaix2009} for an introduction to power law.} meaning that the tail probability $P(X>x)$ decays like a power function $x^{-\alpha}$ for large $x$, where $\alpha>0$ is called the Pareto exponent.\footnote{A more precise statement is that $P(X>x)=x^{-\alpha}\ell(x)$, where $\ell$ is a slowly varying function at infinity. See \cite{BinghamGoldieTeugels1987} for the properties of slowly varying functions and related concepts.} Oftentimes, knowing the Pareto exponent $\alpha$ is of considerable practical interest because it determines the shape of the income distribution for the rich and hence income inequality. As a motivating example, consider the theory of optimal taxation. If the government's objective is to maximize the total tax revenue, then \cite{Saez2001} shows that the optimal income tax rate is $\tau=\frac{1}{1+\alpha e}$, where $\alpha$ is the income Pareto exponent and $e$ is the elasticity of income in the top bracket with respect to the tax rate. If we set $e=0.3$ as estimated by \cite{PikettySaezStancheva2014} and $\alpha=1.5$--$3$ as is often reported,\footnote{See \citet[Table 13A.23]{AtkinsonPiketty2010} for a list of income Pareto exponents across time and countries estimated from top income share data. Studies that use micro data such as \cite{Toda2012JEBO} and \cite{IbragimovIbragimov2018} also find similar numbers.} then the optimal tax rate ranges from 70\% when $\alpha=1.5$ and 53\% when $\alpha=3$. Clearly, the knowledge of the Pareto exponent is important for policy design.

When individual data on income is available, it is relatively straightforward to estimate and conduct inference on the Pareto exponent, either by maximum likelihood \citep{hill1975}, log rank regressions \citep{gabaix-ibragimov2011}, fixed-$k$ asymptotics \citep{MullerWang2017}, or other methods. Even if individual data is not available, if we have binned data we can still estimate the Pareto exponent by eyeballing \citep{Pareto1897Cours} or maximum likelihood \citep{VirkarClauset2014}. However, in practice it is often the case (especially for administrative data) that only some top income shares are reported and individual data are not available. A typical example is Table \ref{t:TopSharesUS} below, which summarizes the U.S.\ household income distribution.\footnote{\label{fn:PS}These numbers are taken from Table A.3 (top income shares including capital gains) of the updated spreadsheet for \cite{piketty-saez2003}, which can be downloaded at \url{https://eml.berkeley.edu/~saez/TabFig2017prel.xls}.} Such income data in the form of tabulations are quite common, including the World Inequality Database.\footnote{\label{fn:WID}\url{https://wid.world/}}

\begin{table}[!htb]
\centering
\caption{U.S.\ top income shares (\%).}\label{t:TopSharesUS}
\begin{tabular}{ccccccc}
\toprule
Year & \multicolumn{6}{c}{Top income percentiles} \\
& 0.01 & 0.1 & 0.5 & 1 & 5 & 10\\
\midrule
1917 & 3.37 & 8.40 & 14.34 & 17.74 & 30.64 & 40.51 \\
\vdots & \vdots & \vdots & \vdots & \vdots & \vdots & \vdots \\
2017& 4.95 & 10.43 & 17.16 & 21.47 & 38.14 & 50.14 \\
\bottomrule
\end{tabular}
\end{table}

Existing studies of such income share data typically rely on some parametric assumption of the \emph{whole} distribution. Given a parametric density, the income shares can be expressed as functions of the unknown parameters and then be estimated by moment-based estimators. Statistical inference is constructed by either using the asymptotic normality or bootstrapping. 
For example, \cite{McDonald1984} and \cite{kleiber-kotz2003} propose to use the generalized beta type-II distribution (GB2) to approximate the whole income distribution.\footnote{However, \cite{Toda2012JEBO,Toda2017MD} show that GB2 is outperformed by the double Pareto-lognormal (dPlN) distribution as a model of income and consumption distributions.} Under this assumption, \cite{ChotikapanichGriffithsRao2007}, \cite{ChotikapanichGriffithsRaoValencia2012}, and \cite{HajargashtGriffithsBriceRaoChotikapanich2012} develop moment-based estimators and inference methods for grouped mean and share data, and \cite{Chen2018} further constructs a unified framework that allows for general form of grouped data. These methods all focus on the mid-sample moments, which can be expressed as known functions of the parameters of the GB2 distribution.\footnote{\label{fn:chen2018}In the application section, \cite{Chen2018} uses Chinese data on ten deciles income shares and U.S.\ data on a series of quintiles, the top 5\% shares, and sample quantiles. These are all mid-sample moments relative to the top income shares considered in the present paper.}

However, the parametric assumption on the whole distribution may lead to a substantial misspecification error when the object of interest is in the tail. This is because tail properties such as very large quantiles are typically in a large scale, and hence a small misspecification in mid-sample can be amplified by a large factor in the tail. For example, the standard normal distribution and the Student-$t$ distribution with degree of freedom 20 share almost the same shape in mid-sample but exhibit substantially different top quantiles. Such misspecification is documented by \cite{Brzezinski2013} in an extensive simulation study.\footnote{\label{fn:distfree}\cite{BeachDavidson1983} and \cite{BeachRichmond1985} propose some distribution-free methods for estimation and inference about the Lorenz curves. Their methods require estimating the population mean, variance, and some other moments, which is not feasible in our situation. Furthermore, their focus is more on the middle sample instead of the tail.}

In addition to the potential misspecification in mid-sample, there is another source of bias when studying tail related objects, which is the dependence among large order statistics. Suppose we are interested in the right tail of the underlying distribution. Then essentially only the largest order statistics are informative. Even if the observations are independent, the largest order statistics are not. Such dependence may incur a large misspecification error again if it is ignored, especially when the very top income share such as 0.01\% is considered. For example, think of the population size as $10^5$. The top 0.01\% share involves only the ten largest order statistics, whose distribution has to be jointly modeled to capture the dependence.

In this paper, we focus only on the top income shares and propose an efficient estimation method for the Pareto exponent. Compared with existing methods, the new estimator takes into consideration of the dependence among large order statistics and is robust to misspecification in mid-sample. In particular, our method is based on the following observations. By definition, top income shares are the ratios between the sum of order statistics for some top percentile and total income. Assuming that the upper tail of the income distribution is Pareto, we derive the joint asymptotic distribution of normalized top income shares using the results on the weighted sums of order statistics due to \cite{Stigler1974}. From this result, we define the classical minimum distance (CMD) estimator \citep{Chiang1956,Ferguson1958} and derive its asymptotic properties.

More specifically, we typically cannot identify the shape of the underlying distribution without observing individual data. However, if we assume the sample size is large enough (but not necessarily known) and the underlying distribution has a Pareto upper tail, we can show that the top shares are jointly asymptotically Gaussian with the mean vector and the variance-covariance matrix being characterized by the Pareto exponent and the scale parameter. Since the scale parameter is not identified given only the shares, we eliminate it by imposing scale invariance and considering a self-normalized statistic, whose distribution is still jointly normal but now fully characterized by the Pareto exponent only. Thus, the problem is asymptotically equivalent to estimating a single parameter in a joint normal distribution using a single random draw from it. The efficient solution is then to consider the continuously updated minimum distance estimator (CUMDE). As we show in simulations, this estimator has excellent finite sample properties when the model is correctly specified.

We note that the Pareto assumption is required for the tail only instead of the whole income distribution, which is why our method is robust to misspecification in mid-sample. In particular, when the data generating process is not exactly Pareto, our estimator still performs well when we only use small enough top percentiles such as the top 1\% and the sample size is large enough, which is typically the case for income share data based on tax returns (where the number of households is in the order of a million). 

We apply our new method to estimate the income Pareto exponent in U.S.\ and France. In U.S., we estimate that the income Pareto exponent has decreased from about 2.2 in 1975 to about 1.6 in 1985, which has remained relatively stable since then at around 1.5 with a conservative 95\% confidence interval of length no more than 0.1. This finding is in stark contrast to other inequality measures such as the top 1\% income share, which has increased from about 10\% in 1985 to 20\% at present, and suggests that the rise in inequality during the last three decades is mainly driven by redistribution between the rich and poor, not among the rich. In France, we find that the Pareto exponent is stable at around 2 postwar.


\section{Weighted sums of order statistics}\label{sec:orderstat}
In this section we derive the asymptotic distribution of the weighted sums of order statistics of a Pareto distribution, which we subsequently use to construct the estimator of the Pareto exponent. 

Let $\set{Y_i}_{i=1}^n$ be independent and identically distributed (i.i.d.) copies of a positive random variable $Y$ with cumulative distribution function (CDF) $F(y)$ and density $f(y)=F'(y)$. Let
$$Y_{(1)}\ge \dots \ge Y_{(n)}$$
denote the order statistics. Following \cite{Stigler1974}, consider the weighted sum
$$L_n=\frac{1}{n}\sum_{i=1}^n J\left(\frac{i}{n+1}\right)Y_{(n-i+1)},$$
where $J:[0,1]\to \R$ is a function that is bounded and continuous almost everywhere with respect to the Lebesgue measure. When
\begin{equation}
J(x)=1[1-q<x\le 1-p]\label{eq:defJ}
\end{equation}
for some $0<p<q\le 1$, $L_n$ can be interpreted as the sum of $Y_{(i)}$'s between the top $100p$ and $100q$ percentiles, divided by the sample size $n$.

The following lemma shows that $L_n$ is asymptotically normal.

\begin{lem}\label{lem:Stigler}
Let $J$ be as in \eqref{eq:defJ}. Then
$$\sqrt{n}(L_n-\mu(J,F))\dto N(0,\sigma^2(J,F)),$$
where
\begin{subequations}\label{eq:momentJF}
\begin{align}
\mu(J,F)&=\int_0^1 J(x)F^{-1}(x)\diff x,\label{eq:muJF}\\
\sigma^2(J,F)&=\int_0^1\int_0^1 \frac{J(x_1)J(x_2)}{f(F^{-1}(x_1))f(F^{-1}(x_2))}(\min\set{x_1,x_2}-x_1x_2)\diff x_1\diff x_2.\label{eq:sigJF}
\end{align}
\end{subequations}
\end{lem}

\begin{proof}
The statement follows from \citet[Theorem 5]{Stigler1974} and the change of variable $x=F(y)$. Note that $J(x)=1[1-q<x\le 1-p]$ implies $J(x)=0$ for $x>1-p$.
\end{proof}

In the remainder of the paper, we assume that $Y$ is Pareto distributed with Pareto exponent $\alpha>1$ and minimum size $c>0$, so $F(y)=1-(y/c)^{-\alpha}$ for $y\ge c$. The Pareto exponent $\alpha$ captures the shape and the minimum size $c$ characterizes the scale. By simple algebra, we obtain
\begin{subequations}
\begin{align}
f(y)&=F'(y)=\alpha c^\alpha y^{-\alpha-1},\label{eq:fPareto}\\
F^{-1}(x)&=c(1-x)^{-1/\alpha},\label{eq:FinvPareto}\\
f(F^{-1}(x))&=\frac{\alpha}{c}(1-x)^{1+1/\alpha}.\label{eq:fFinvPareto}
\end{align}
\end{subequations}
When $Y$ is Pareto distributed, we can explicitly compute the moments in Lemma \ref{lem:Stigler} as follows.

\begin{lem}\label{lem:StiglerPareto}
Let $J$ be as in \eqref{eq:defJ} and $F$ be the Pareto CDF with exponent $\alpha>1$ and minimum size $c$. Letting $\xi=1/\alpha<1$, we have
\begin{subequations}\label{eq:momentpq}
\begin{align}
&\mu(J,F)=\mu(p,q)\coloneqq c\frac{q^{1-\xi}-p^{1-\xi}}{1-\xi},\label{eq:mupq}\\
&\sigma^2(J,F)=\sigma^2(p,q)\notag\\
&\coloneqq \frac{2c^2\xi^2}{1-\xi}\left(\frac{q^{1-2\xi}-p^{1-2\xi}}{1-2\xi}+p^{1-\xi}\frac{q^{-\xi}-p^{-\xi}}{\xi}+\frac{2p^{1-\xi}q^{1-\xi}-p^{2-2\xi}-q^{2-2\xi}}{2-2\xi}\right)\label{eq:sigpq},
\end{align}
\end{subequations}
where $\frac{q^{1-2\xi}-p^{1-2\xi}}{1-2\xi}$ is interpreted as $\log \frac{q}{p}$ if $\xi=1/2$.
\end{lem}

Next, we consider the joint distribution of the sums of $Y_{(i)}$'s over some top percentile groups. Suppose that there are $K$ groups indexed by $k=1\dots,K$, and the $k$-th group corresponds to the top $p_k$ to $p_{k+1}$ percentile, where $0<p_1<\dots<p_K<p_{K+1}\le 1$. Define
\begin{equation}
\bar{Y}_k=\frac{1}{n}\sum_{i=\floor{np_k}+1}^{\floor{np_{k+1}}}Y_{(i)},\label{eq:Ybar}
\end{equation}
where $\floor{x}$ denotes the largest integer not exceeding $x$.\footnote{\label{fn:Bin0}We exclude the largest $\floor{np_1}$ order statistics since their average does not satisfy the assumptions of Central Limit Theorem when $\alpha<2$.} By Lemmas \ref{lem:Stigler} and \ref{lem:StiglerPareto}, we have
$$\sqrt{n}(\bar{Y}_k-\mu_k)\dto N(0,\sigma_k^2),$$
where $\mu_k=\mu(p_k,p_{k+1})$ and $\sigma_k^2=\sigma^2(p_k,p_{k+1})$ are given by \eqref{eq:mupq} and \eqref{eq:sigpq}, respectively. Let $\bar{Y}=(\bar{Y}_1,\dots,\bar{Y}_K)^\top$ and $\mu=(\mu_1,\dots,\mu_K)^\top$. Then by the Cram\'er-Wold device, it follows that
\begin{equation}
\sqrt{n}(\bar{Y}-\mu)\dto N(0,\Sigma),\label{eq:multnormal}
\end{equation}
where $\Sigma$ is some variance matrix with $\Sigma_{kk}=\sigma_k^2$. The following lemma gives an explicit formula for $\Sigma$.

\begin{lem}\label{lem:Sigma}
The variance matrix $\Sigma$ in \eqref{eq:multnormal} is symmetric and
\begin{equation}
\Sigma_{jk}=\begin{cases}
\sigma_k^2=\sigma^2(p_k,p_{k+1}), & (j=k)\\
-c^2\xi^2\frac{p_{j+1}^{1-\xi}-p_j^{1-\xi}}{1-\xi}\left(\frac{p_{k+1}^{-\xi}-p_k^{-\xi}}{\xi}+\frac{p_{k+1}^{1-\xi}-p_k^{1-\xi}}{1-\xi}\right). & (j<k)
\end{cases}\label{eq:Sigjk}
\end{equation}
Furthermore, $\Sigma$ is positive definite.
\end{lem}

\section{Minimum distance estimator}\label{sec:CUMDE}

In practice, the income distribution is often presented as a tabulation of top income shares as in Table \ref{t:TopSharesUS} and micro data is not available. In this case the researcher is forced to conduct inference on the Pareto exponent $\alpha$ based on the top income shares of the given top percentiles, for example
\begin{equation}
p=(p_1,p_2,p_3,p_4,p_5,p_6)=\frac{1}{100}(0.01,0.1,0.5,1,5,10) \label{eq:TopPercent}
\end{equation}
as in Table \ref{t:TopSharesUS}. If $Y$ is distributed as Pareto with exponent $\alpha>1$ and minimum size $c>0$, using $F(y)=1-(y/c)^{-\alpha}$, the population top $p$ percentile is
$$1-(y/c)^{-\alpha}=1-p\iff y=cp^{-1/\alpha}.$$
Using \eqref{eq:fPareto}, the total income held by the population top $p$ percentile is
$$Y(p)\coloneqq \int_{cp^{-1/\alpha}}^\infty y\alpha c^\alpha y^{-\alpha-1}\diff y=c\frac{\alpha}{\alpha-1}p^{1-1/\alpha}.$$
Therefore the population top $p$ income share is
$$S(p)\coloneqq Y(p)/Y(1)=p^{1-1/\alpha},$$
which depends only on $p$ and $\alpha$. If $Y$ is Pareto only for the upper tail, a similar calculation yields
\begin{equation}
S(p)/S(q)=(p/q)^{1-1/\alpha}\iff \alpha=\frac{1}{1-\frac{\log (S(q)/S(p))}{\log (q/p)}}\label{eq:alphaSpq}
\end{equation}
for $0<p<q\ll 1$. \citet[Table 13A.23]{AtkinsonPiketty2010} and \citet[Figure 3]{AokiNirei2017} estimate the income Pareto exponents from \eqref{eq:alphaSpq} using $p=0.1\%$ and $q=1\%$.\footnote{\cite{Kuznets1953} and \cite{FeenbergPoterba1993} use similar methods to estimate the Pareto exponent.} 
A natural question is whether such a method can be statistically justified for the tabulation data as in Table \ref{t:TopSharesUS}. In this section, we derive such an estimator and discuss its asymptotic properties.

\subsection{Asymptotic theory}
Let $\set{Y_i}_{i=1}^n$ be the (unobserved) income data and $Y_{(1)}\ge \dots \ge Y_{(n)}$ the order statistics. Let $K\ge 2$ and suppose that some top percentiles $0<p_1<\dots<p_K<p_{K+1}\le 1$ and the corresponding sample top income shares
$$S_k=\frac{\sum_{i=1}^{\floor{np_k}}Y_{(i)}}{\sum_{i=1}^nY_{(i)}},\quad  k=1,\dots,K+1,$$
are given. Suppose that $p_{K+1}$ is small enough such that for $i\le \floor{np_{K+1}}$, we may assume that $Y_{(i)}$ are realizations from a Pareto distribution with exponent $\alpha$ and minimum size $c$. To construct an estimator of $\alpha$ based only on $\set{S_k}$, we consider the vector of self-normalized non-overlapping top income shares defined by
\begin{equation}
\bar{s}=(\bar{s}_1,\dots,\bar{s}_{K-1})^\top\coloneqq \left(\frac{S_2-S_1}{S_{K+1}-S_K},\dots,\frac{S_K-S_{K-1}}{S_{K+1}-S_K}\right)^\top.\label{eq:normshares}
\end{equation}
The following proposition shows that $\bar{s}$ is asymptotically normal.

\begin{prop}\label{prop:normshares}
Let $r_k=\mu_k/\mu_K$, where $\mu_k=\mu(p_k,p_{k+1})$ is given by \eqref{eq:mupq}. Define the $(K-1)$-vector $r=(r_1,\dots,r_{K-1})^\top$ and $(K-1)\times K$ matrix $H=\begin{bmatrix}I_{K-1} & -r\end{bmatrix}/\mu_K$. Then
$$\sqrt{n}(\bar{s}-r)\dto N(0,H\Sigma H^\top).$$
The variance matrix $\Omega(\alpha)\coloneqq H\Sigma H^\top$ depends only on $\alpha$ and is positive definite.
\end{prop}

Based on Proposition \ref{prop:normshares}, it is natural to consider the classical minimum distance (CMD) estimator \citep{Chiang1956,Ferguson1958}
\begin{equation}
\widehat{\alpha}=\argmin_{\alpha\in A}(r(\alpha)-\bar{s})^\top \widehat{W} (r(\alpha)-\bar{s}),\label{eq:MDE}
\end{equation}
where $\widehat{W}$ is some symmetric and positive definite weighting matrix and $A$ is some compact parameter space.

Let $G_n(\alpha)$ be the objective function in \eqref{eq:MDE}. Suppose that $\widehat{W}\pto W$ as $n\to\infty$, where $W$ is also positive definite. Letting $\alpha_0\in A$ be the true Pareto exponent, we have
$$G_n(\alpha)\pto G(\alpha)\coloneqq (r(\alpha)-r(\alpha_0))^\top W(r(\alpha)-r(\alpha_0)).$$
Since $W$ is positive definite, we have $G(\alpha)\ge 0=G(\alpha_0)$, with equality if and only if $r(\alpha)=r(\alpha_0)$. The following proposition shows that the parameter $\alpha$ is point-identified by this condition.

\begin{prop}[Identification]\label{prop:identify}
$\alpha\neq \alpha_0$ implies $r(\alpha)\neq r(\alpha_0)$.
\end{prop}

Using standard arguments, consistency and asymptotic normality follows from the above identification result.

\begin{thm}[Consistency]\label{thm:consistent}
Let $A\subset (1,\infty)$ be compact, $\alpha_0\in A$, and suppose $\widehat{W}\pto W$ as $n\to\infty$, where $W$ is positive definite. Let $\widehat{\alpha}$ be the minimum distance estimator in \eqref{eq:MDE}. Then $\widehat{\alpha}\pto \alpha_0$.
\end{thm}

\begin{proof}
Clearly $G(\alpha)$ is continuous in $\alpha>1$. The statement follows from Proposition \ref{prop:identify}, the uniform law of large numbers, and \citet[Theorem 2.1]{NeweyMcFadden1994}.
\end{proof}

\begin{thm}[Asymptotic normality]\label{thm:asymptotic}
Let $r(\alpha), \Omega(\alpha)$ be defined as in Proposition \ref{prop:normshares}. Suppose that the assumptions of Theorem \ref{thm:consistent} hold and $\alpha_0$ is an interior point of $A$. Then
$$\sqrt{n}(\widehat{\alpha}-\alpha_0)\dto N(0,V)$$
as $n\to \infty$, where
$$V=(R^\top WR)^{-1}R^\top W\Omega WR(R^\top WR)^{-1}$$
for $\Omega=\Omega(\alpha_0)$ and $R=\nabla r(\alpha_0)$.
\end{thm}

\begin{proof}
Immediate from Theorem \ref{thm:consistent} and \citet[Theorem 3.2]{NeweyMcFadden1994}.
\end{proof}

By standard results in classical minimum distance estimation \citep{Chiang1956,Ferguson1958}, we achieve efficiency by choosing the weighting matrix such that $\widehat{W}\pto W=\Omega^{-1}$. Therefore the most natural estimator is the following continuously updated minimum distance estimator (CUMDE).

\begin{cor}[Efficient CMD]\label{cor:CUMDE}
Let everything be as in Theorem \ref{thm:asymptotic} and define the continuously updated minimum distance estimator (CUMDE) by
\begin{equation}
\widehat{\alpha}=\argmin_{\alpha\in A}(r(\alpha)-\bar{s})^\top\Omega(\alpha)^{-1}(r(\alpha)-\bar{s}),\label{eq:CUMDE}
\end{equation}
where $\Omega(\alpha)$ is given as in Proposition \ref{prop:normshares}. Then
\begin{equation}
\sqrt{n}(\widehat{\alpha}-\alpha_0)\dto N(0,(R^\top \Omega^{-1}R)^{-1}),\label{eq:CUMDE_lim}
\end{equation}
where $\Omega=\Omega(\alpha_0)$ and $R=\nabla r(\alpha_0)$. The estimator $\widehat{\alpha}_{\mathrm{CUMDE}}$ has the minimum asymptotic variance among all CMD estimators.
\end{cor}
We can use \eqref{eq:CUMDE_lim} to construct confidence intervals of $\alpha$.

We now consider testing the null hypothesis $H_0$: $\alpha=\alpha_0$ against the alternative $H_1$: $\alpha\neq \alpha_0$. The following propositions show that we can implement likelihood ratio and specification tests, which avoid computing the derivative of $r(\alpha)$. We omit the proofs since they are analogous to standard GMM results \citep[Section 9]{NeweyMcFadden1994}. The likelihood ratio test can also be inverted to construct the confidence interval. 

\begin{prop}[Likelihood ratio test]\label{prop:LR}
Under the null $H_0$: $\alpha=\alpha_0$, we have
$$n(G_n(\alpha)-G_n(\widehat{\alpha}))\dto \chi^2(1).$$
Under the alternative $H_1$: $\alpha\neq \alpha_0$, we have
$$n(G_n(\alpha)-G_n(\widehat{\alpha}))\pto \infty.$$
\end{prop}

\begin{prop}[Specification test]\label{prop:spec}
Suppose that $K\ge 3$. If $F$ is the Pareto CDF with some exponent $\alpha\in A$, then
$$nG_n(\widehat{\alpha})\dto \chi^2(K-2).$$
\end{prop}

\subsection{Discussion and implementation}\label{subsec:implement}

In this section we discuss the choice of the top income shares and implementation of our estimation method. As in Section \ref{sec:orderstat}, let $Y$ be a positive random variable with cumulative distribution function $F$. Note that our Pareto assumption serves as a tail approximation of the underlying distribution $F$, which can actually be substantially different from the exact Pareto distribution. Such approximation has been formally justified in the statistic literature under very mild primitive assumptions. To be specific, consider some tail cutoff $u$ and define 
$$F_{u}(y) =\frac{F(u+y)-F(u)}{1-F(u)}$$
as the conditional CDF given $Y\ge u$. Also define the generalized Pareto distribution (GPD, \citealp[Chapter 3]{deHaanFerreira2006}), which is given by
\begin{equation}
G(y;\xi,\sigma)=\begin{cases}
1-\left(1+\xi\frac{y}{\sigma}\right)^{-1/\xi}, & (\xi\neq 0)\\
1-\exp(-y/\sigma), & (\xi=0)
\end{cases}\label{GPD}
\end{equation}
with $y\ge 0$ if $\xi\ge 0$ and $y\in (0,-\sigma/\xi)$ otherwise. Let $y_0$ be the right end-point of the support of $Y$, which is $\infty$ if $\xi\ge 0$. Then the Pickands-Balkema-de Haan Theorem \citep{BalkemaDeHaan1974,Pickands1975} states that the GPD is a good approximation of $F_{u}$ in the sense that
\begin{equation}
\lim_{u\to y_0}\sup_{y\in (0,y_0-u)}\abs{F_{u}(y) -G(y;\xi,\sigma)}=0, \label{GPDapprox}
\end{equation}
where the scale parameter $\sigma>0$ implicitly depends on $u$. The parameter $\alpha=1/\xi$ is our object of interest. It is uniquely determined by the underlying distribution and characterizes its tail heaviness.

The necessary and sufficient condition for the approximation \eqref{GPDapprox} to hold is that $F$ lies in the domain of attraction of one of the three limit laws, which is a mild condition and holds for almost all commonly used distributions. The positive $\alpha$ case covers distributions with a Pareto-type tail such as Pareto, Student-$t$, and $F$ distributions. In particular, if $F$ is the standard Pareto distribution such that $1-F(y)=y^{-\alpha }$, then $F_u(y)=G(y;\xi,\sigma)$ holds with $\xi=1/\alpha$ and $\sigma=u/\alpha$. If $F$ is Student-$t$ distribution with $\nu$ degrees of freedom, then \eqref{GPDapprox} holds asymptotically as $u$ diverges with $\xi$ being equal to $1/\nu$. See \citet[Chapter 1]{deHaanFerreira2006} for an overview.

For the estimation of $\alpha$, the practical determination of $u$ (and our top income percentiles $p_1,\dots,p_{K+1}$) is widely accepted as a difficult question even when $\set{Y_i}_{i=1}^n$ is observed. It becomes more challenging (if possible at all) in our setting with tabulations. To see this, consider the example that $F$ is a mixture of some Pareto distribution with probability 0.1 and the standard normal distribution with probability 0.9. This structure implies that only the very few top shares are informative about the true tail. In this case, choosing too many top shares, say up to 10\%, would implicitly include too many observations from the mid-sample, which incurs a large bias. However, choosing fewer top shares leads to fewer observations and hence compromises the asymptotic Gaussianity of the central limit theorem. In principle, there cannot exist a procedure that consistently determines the optimal choice of $u$ since $F$ is unknown. This is close in spirit to the bias-variance trade-off in the choice of bandwidth in standard kernel regressions. \citet[Theorem 5.1]{MullerWang2017} formalize this result in the case with full observations. Given this difficulty, we resort to simulation studies in 
Appendix \ref{sec:sim} for the selection of top shares in the application in Section \ref{sec:appl}.

By Corollary \ref{cor:CUMDE}, we can compute $\widehat{\alpha}$ by numerically solving the minimization problem \eqref{eq:CUMDE}. However, it is clear from Lemmas \ref{lem:StiglerPareto} and \ref{lem:Sigma} that $\xi=1/\alpha$ shows up everywhere in $r(\alpha)$ and $\Omega(\alpha)$, and hence it is more convenient to optimize over $\xi=1/\alpha\in (0,1)$ instead of $\alpha>1$. Therefore let $\tilde{r}(\xi)=r(1/\xi)$ and $\tilde{\Omega}(\xi)=\Omega(1/\xi)$ . We can thus estimate $\xi$ (and $\alpha$) using the following algorithm.
\begin{enumerate}
\item Given the top income share data $S_1,\dots,S_{K+1}$ for the top $p_1,\dots p_{K+1}$ percentiles, define the normalized shares $\bar{s}$ by \eqref{eq:normshares}.
\item For $\xi\in (0,1)$, define $\tilde{r}_k(\xi)=\frac{p_{k+1}^{1-\xi}-p_k^{1-\xi}}{p_{K+1}^{1-\xi}-p_K^{1-\xi}}$ and $\tilde{r}(\xi)=(\tilde{r}_1(\xi),\dots,\tilde{r}_{K-1}(\xi))^\top$.
\item Define $\tilde{\Omega}(\xi)=\Omega(1/\xi)$ using \eqref{eq:mupq}, \eqref{eq:sigpq}, \eqref{eq:Sigjk}, and Proposition \ref{prop:normshares}, where we can set $c=1$ without loss of generality.
\item Define the objective function
$$\tilde{G}(\xi)=(\tilde{r}(\xi)-\bar{s})^\top \tilde{\Omega}(\xi)^{-1}(\tilde{r}(\xi)-\bar{s})$$
and compute the minimizer $\widehat{\xi}$ of $\tilde{G}$ over $\xi\in (0,1)$. The point estimate of the Pareto exponent $\alpha$ is $\widehat{\alpha}=1/\widehat{\xi}$.
\item If the sample size $n$ is known, use Corollary \ref{cor:CUMDE} or Proposition \ref{prop:LR} to construct the confidence interval. For this purpose, one can use
\begin{align*}
\frac{\tilde{r}_k'(\xi)}{\tilde{r}_k(\xi)}&=\frac{p_{K+1}^{1-\xi}\log p_{K+1}-p_K^{1-\xi}\log p_K}{p_{K+1}^{1-\xi}-p_K^{1-\xi}}-\frac{p_{k+1}^{1-\xi}\log p_{k+1}-p_k^{1-\xi}\log p_k}{p_{k+1}^{1-\xi}-p_k^{1-\xi}},\\
r_k'(\alpha)&=\frac{\diff}{\diff \alpha}\tilde{r}_k(1/\alpha)=-\tilde{r}_k'(\xi)\xi^2.
\end{align*}
\end{enumerate}

In 
Appendix \ref{sec:sim}, we conduct simulation studies and find that our estimator has excellent finite sample properties.

\section{Pareto exponents in U.S.\ and France}\label{sec:appl}

As an empirical application, we estimate the Pareto exponent of the income distribution in U.S.\ for the period 1917--2017 and France for 1900--2014. For U.S., we use the updated top income share data (including capital gains) from \cite{piketty-saez2003} (see Footnote \ref{fn:PS} for details). For France, we obtain the top income shares from the World Inequality Database (Footnote \ref{fn:WID}). These datasets are based on tax returns (administrative data) and underreporting may not be as big an issue as in survey data. We focus on these two countries because long time series of detailed top income shares (top 0.01\%--10\%) are available.

Figure \ref{fig:TopShares} plots the top 1\% and 10\% income shares (including capital gains) for U.S. As is well-known, the series are roughly parallel and exhibit a U-shaped pattern over the century. Figure \ref{fig:Pareto} plots the Pareto exponent estimated as in Section \ref{subsec:implement}. ``Top 1\%'' uses the top 0.01\%, 0.1\%, 0.5\%, and 1\% groups ($K=3$), whereas ``Top 10\%'' also includes the top 5\% and 10\% groups ($K=5$).  For comparison, we also calculate the ``simple'' estimator in \eqref{eq:alphaSpq} using the top 0.1\% and 1\% shares as is common in the literature.

We can make a few observations from Figure \ref{fig:Pareto}. First, the Pareto exponent estimates are significantly different when using the top 1\% and 10\% groups. Based on the discussion in Section \ref{subsec:implement} and the simulation results in 
Appendix \ref{sec:sim}, this suggests that the income distribution is not exactly Pareto and that the 10\% result is biased. Therefore we should focus on the top 1\% result. The Pareto exponent ranges from 1.34 to 2.29. Second, our minimum distance estimator using the top 1\% and the ``simple'' estimator in \eqref{eq:alphaSpq} based on the top 0.1\% and 1\% shares behave similarly. However, according to the simulation results, the minimum distance estimator has better finite sample properties. Finally, Figures \ref{fig:TopShares} and \ref{fig:Pareto} tell different stories about income inequality. While the top 1\% income share in Figure \ref{fig:TopShares} has been rising roughly linearly since about 1975, the Pareto exponent in Figure \ref{fig:Pareto} sharply declines (implying increased inequality) between about 1975 and 1985 but remains flat since then. This observation suggests that the rise in inequality since 1985 as seen in Figure \ref{fig:TopShares} is mainly driven by the redistribution between the rich (top 1\%) and the poor (bottom 99\%), and there is no evidence of increased inequality among the rich.

\begin{figure}[!htb]
\centering
\begin{subfigure}{0.48\linewidth}
\includegraphics[width=\linewidth]{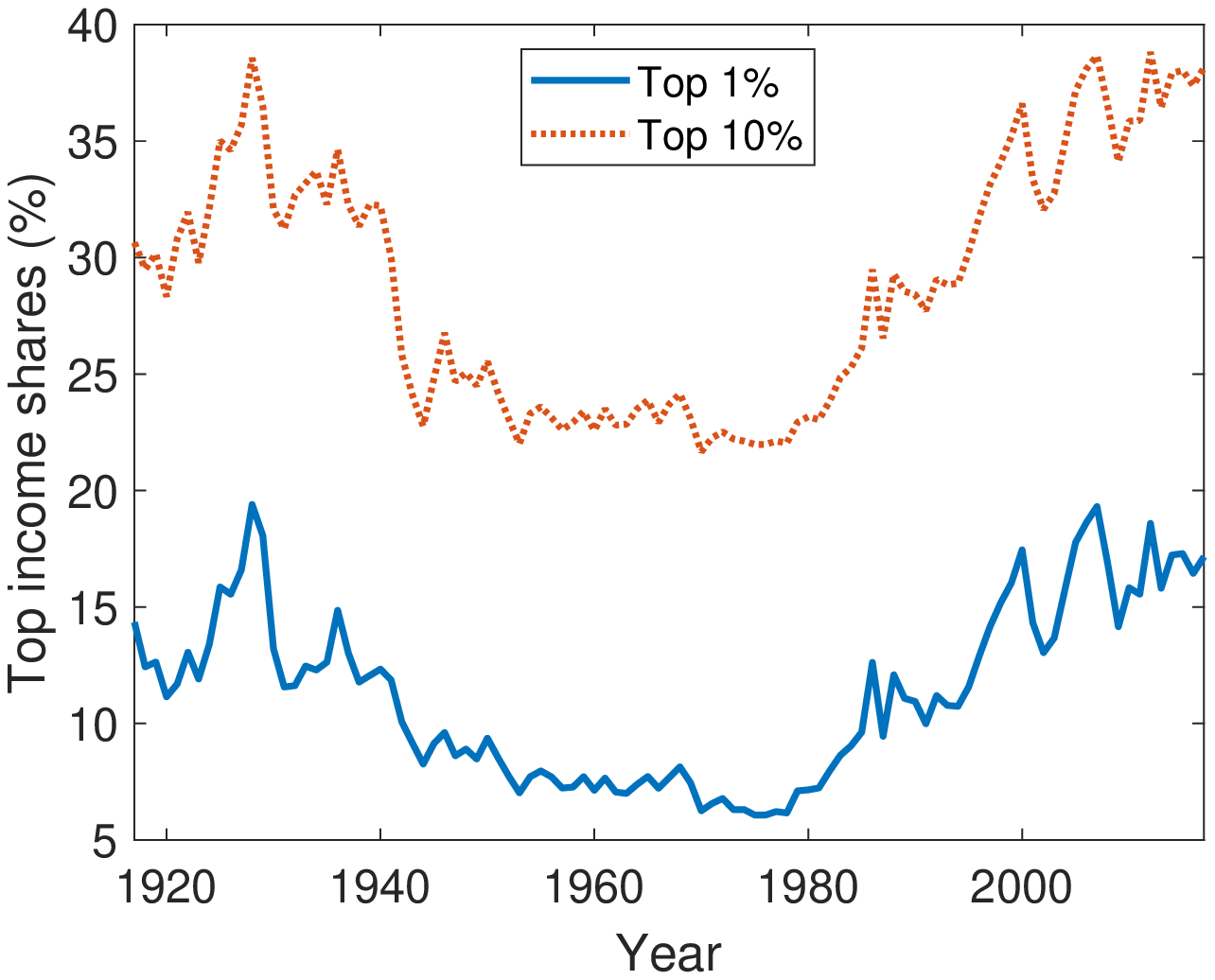}
\caption{Top income shares.}\label{fig:TopShares}
\end{subfigure}
\begin{subfigure}{0.48\linewidth}
\includegraphics[width=\linewidth]{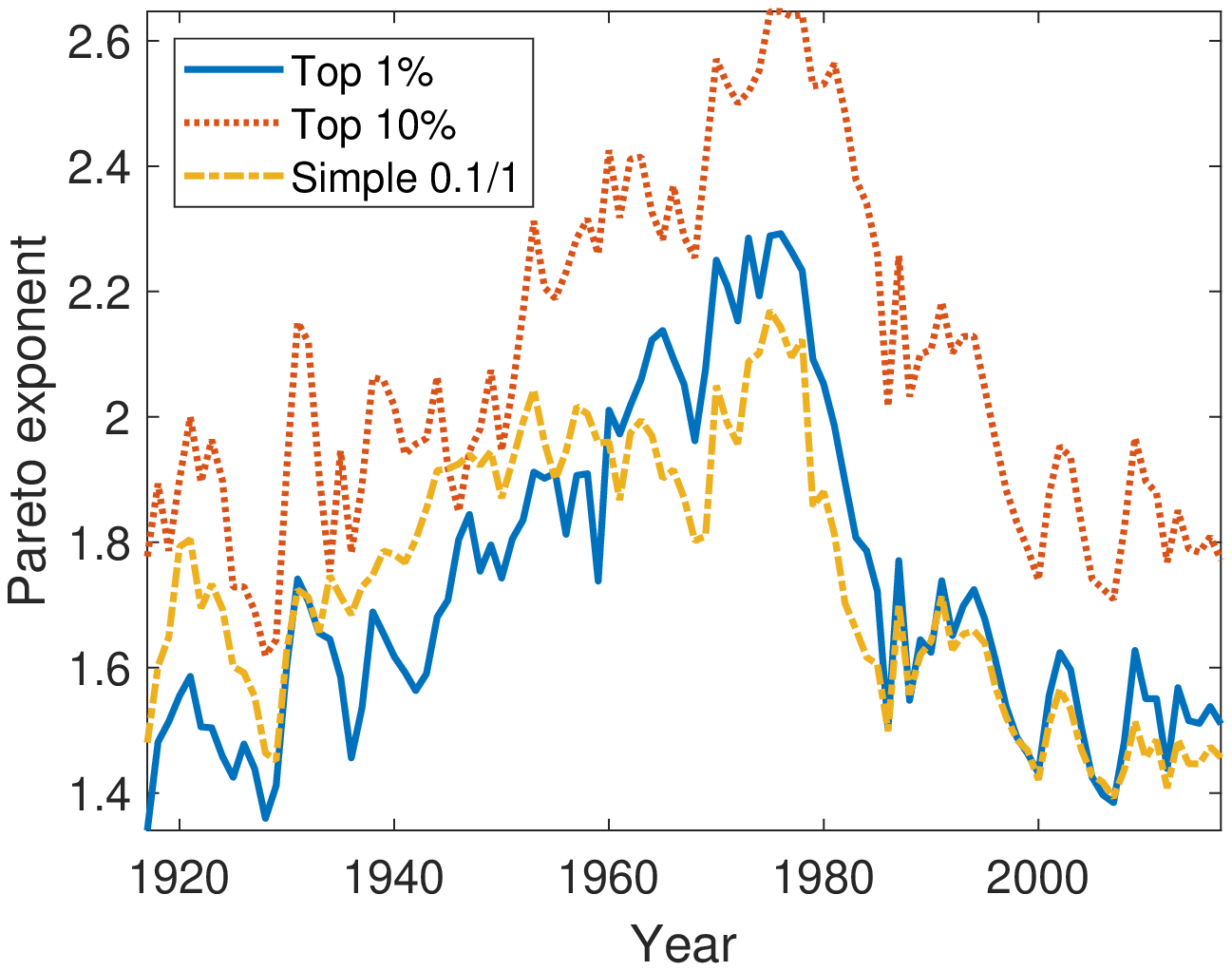}
\caption{Pareto exponent.}\label{fig:Pareto}
\end{subfigure}
\caption{Income distribution in U.S.}\label{fig:incomeUS}
\end{figure}

Figure \ref{fig:incomeFr} repeats the analysis for France. Again, the point estimates of the Pareto exponent when using the top 1\% and 10\% groups differ significantly, and therefore we should focus on the 1\% result. Unlike in U.S., where 1960--1980 appears to be an unusual period of low inequality (high Pareto exponent), in France the Pareto exponent is relatively stable at around 1.5 prewar and 2 postwar. Therefore there seems to be a regime change at around World War II, corroborating to \cite{Piketty2003}'s analysis.

\begin{figure}[!htb]
\centering
\begin{subfigure}{0.48\linewidth}
\includegraphics[width=\linewidth]{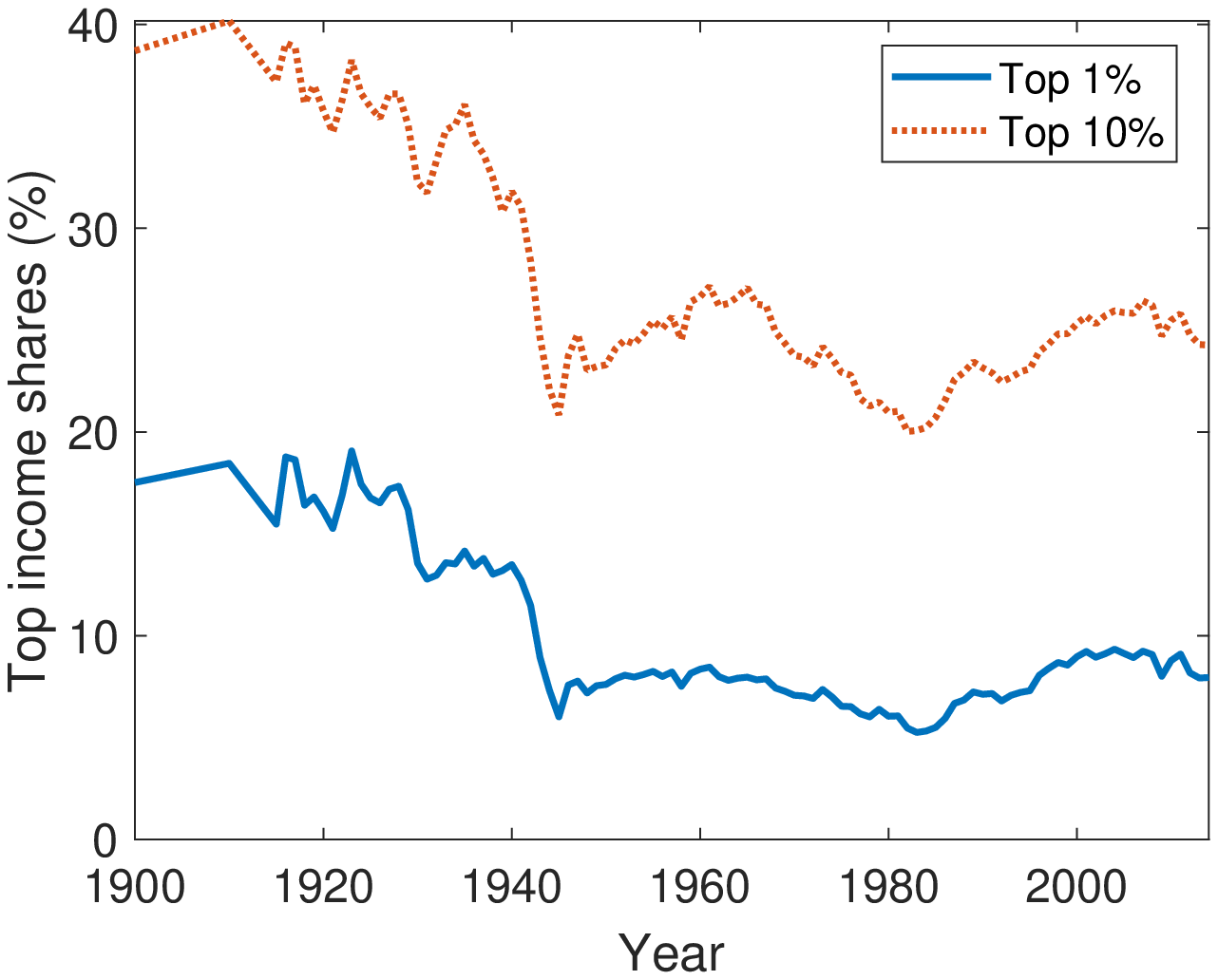}
\caption{Top income shares.}\label{fig:TopSharesFr}
\end{subfigure}
\begin{subfigure}{0.48\linewidth}
\includegraphics[width=\linewidth]{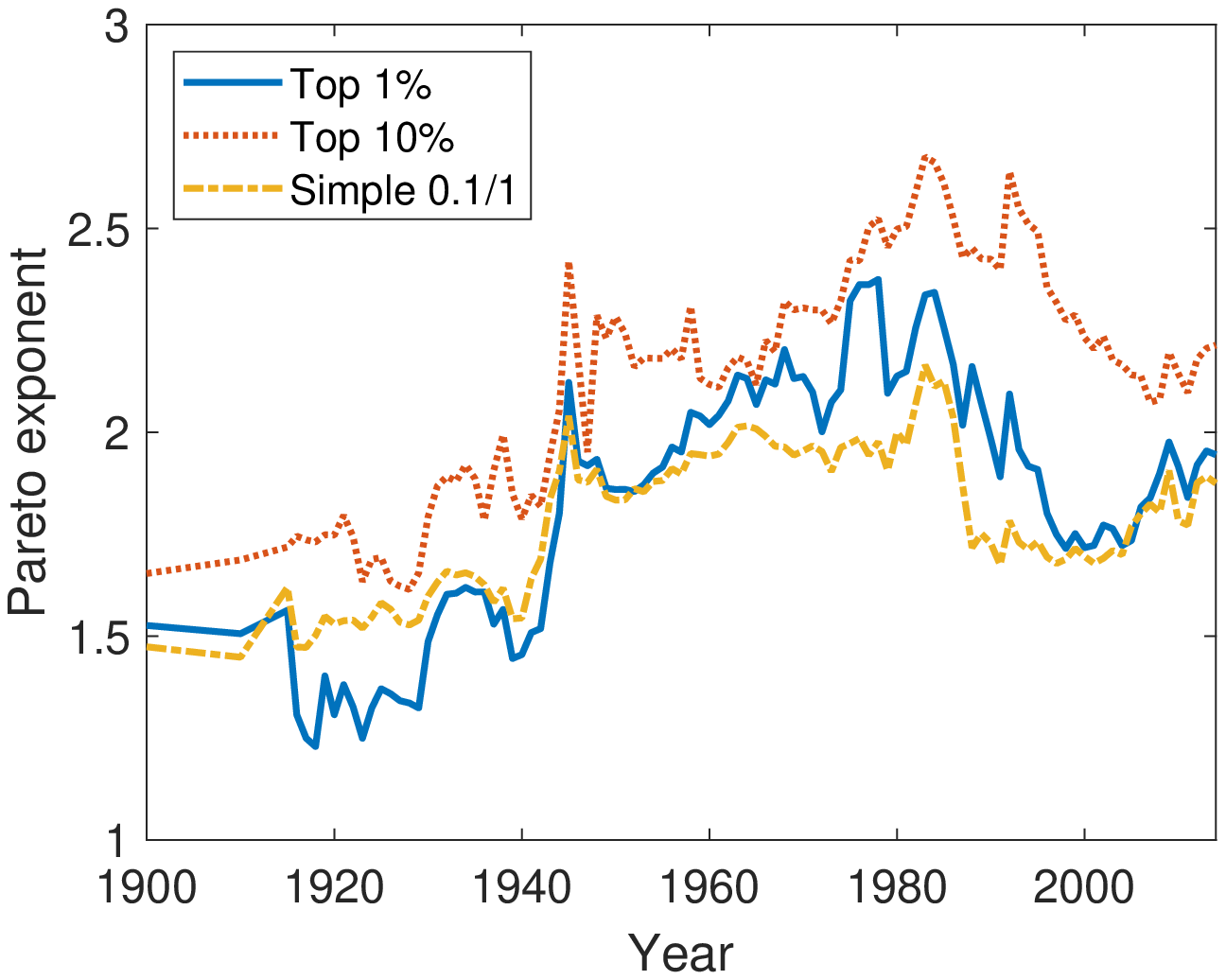}
\caption{Pareto exponent.}\label{fig:ParetoFr}
\end{subfigure}
\caption{Income distribution in France.}\label{fig:incomeFr}
\end{figure}

Conducting statistical inference on $\alpha$ typically requires the sample size $n$ if the dataset is cross-sectional. In our dataset, which consists only of the top income shares, the sample size is unknown. However, there are potentially two approaches to conduct statistical inference. One is to assume a conservative number for the sample size, and the other is to exploit the panel data structure to construct the confidence interval without the knowledge of $n$ by using the method proposed by \cite{IbragimovMueller2010,IbragimovMueller2016} and \citet[Section 3.3]{IbragimovIbragimovWalden2015}.\footnote{We thank an anonymous referee for this suggestion.}

With the first approach, it is reasonable to assume that the number of households is at least one million ($n=10^6$) for both U.S.\ and France. We can use this number to construct conservative confidence intervals as discussed in Section \ref{subsec:implement}. Figure \ref{fig:CI95} shows the length of these conservative 95\% confidence intervals, which are constructed using the asymptotic normality of $\widehat{\alpha}$ as in Corollary \ref{cor:CUMDE}. For both U.S.\ and France, the case when using only the income shares within the top 1\% (which is more relevant due to possible model misspecification), the length is at most around 0.1, which is similar to the number in Table \ref{t:sim} with sample size $n=10^6$. Therefore the confidence intervals are within $\pm 0.05$ of the estimated Pareto exponents. For example, from Figures \ref{fig:Pareto}, \ref{fig:ParetoFr}, and \ref{fig:CI95}, we can conclude that the income Pareto exponent has significantly declined from late 1970s to early 2000s both in U.S.\ and France, although in U.S.\ the Pareto exponent has been stable at around 1.5 since 1985. Returning to the optimal taxation problem discussed in the introduction, this estimate together with income elasticity 0.3 suggests that the revenue-maximizing income tax rate is $1/(1+1.5\times 0.3)=69\%$.

\begin{figure}[!htb]
\centering
\begin{subfigure}{0.48\linewidth}
\includegraphics[width=\linewidth]{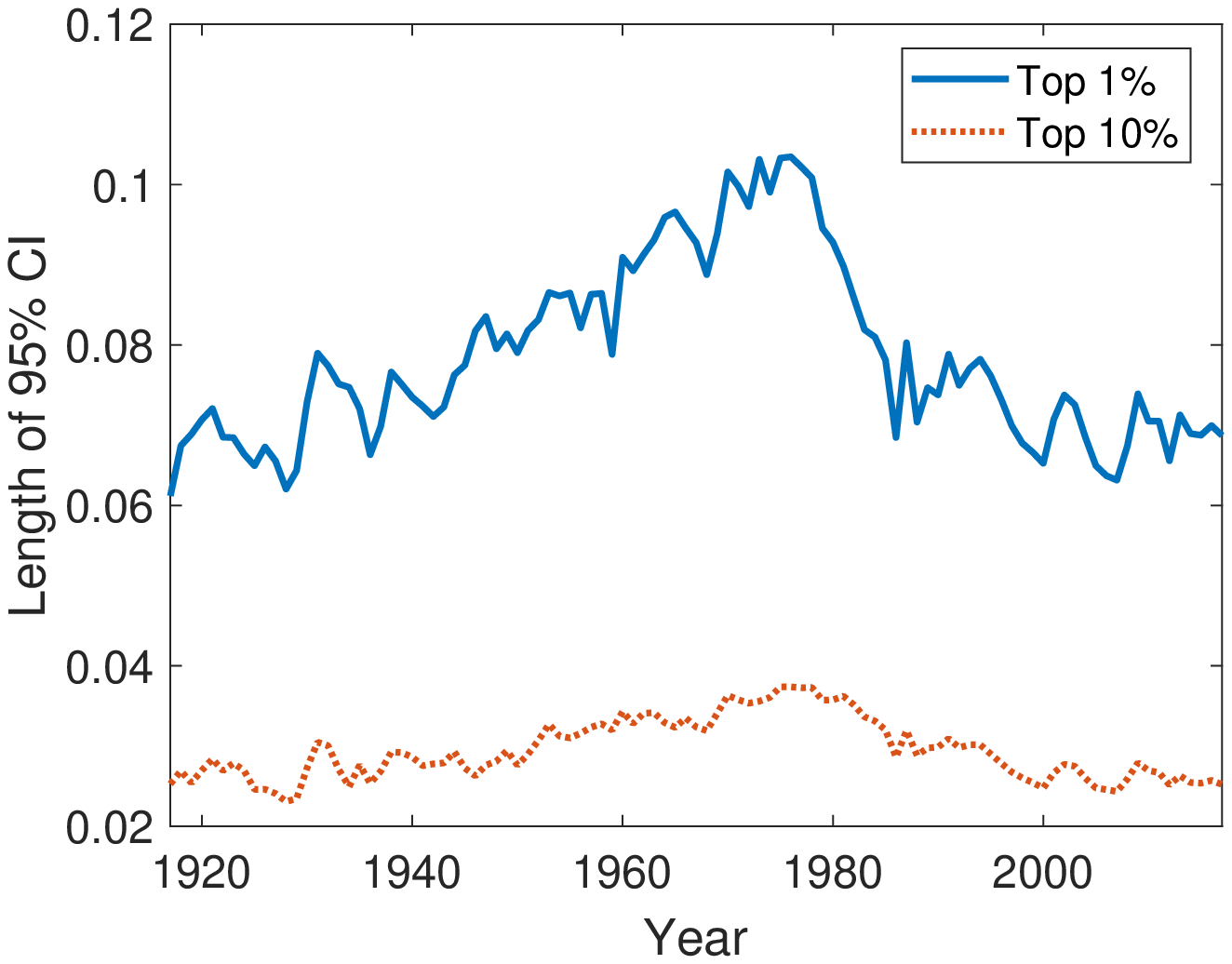}
\caption{U.S.}\label{fig:CI95US}
\end{subfigure}
\begin{subfigure}{0.48\linewidth}
\includegraphics[width=\linewidth]{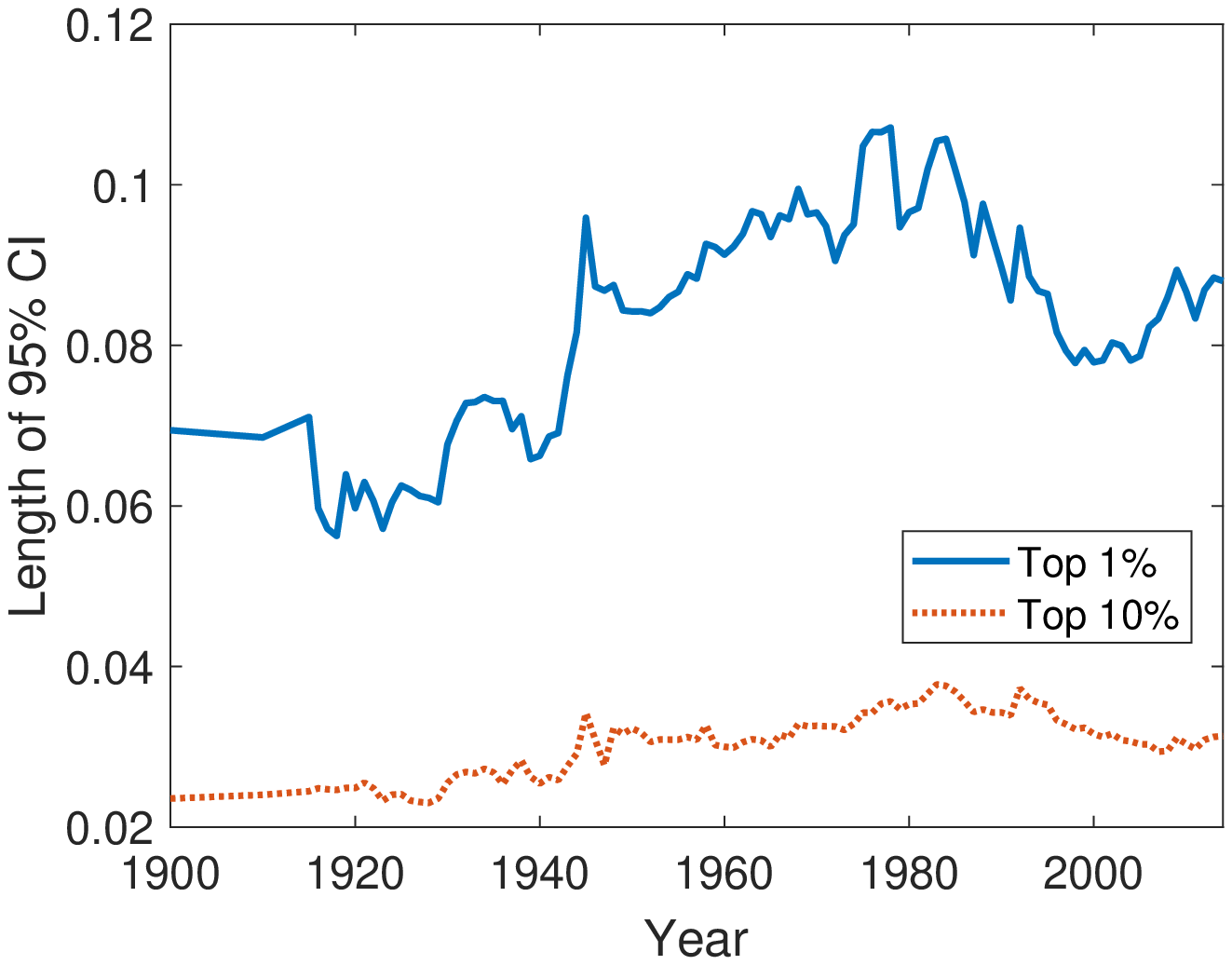}
\caption{France.}\label{fig:CI95Fr}
\end{subfigure}
\caption{Length of conservative 95\% confidence intervals of Pareto exponents.}\label{fig:CI95}
\end{figure}

With the second approach, consider $\set{\widehat{\alpha}_1,\dots,\widehat{\alpha}_L}$ as $L$ estimators of $\alpha$ using $L$ independent samples, where the sample sizes can be heterogeneous. Suppose the estimators satisfy asymptotic normality such that for all $l=1,\dots,L$, we have
\begin{equation*}
\sqrt{n_l}( \widehat{\alpha}_l-\alpha ) \dto N( 0,\sigma _l^{2})
\end{equation*}
as $n_l\to \infty$, where $n_l$ denotes the sample size of the $l$-th sample. The usual $t$-statistic for the hypothesis testing problem that 
\begin{equation*}
H_{0}:\alpha =\alpha_{0}\quad \text{against}\quad H_{1}:\alpha \neq \alpha_{0}
\end{equation*}
is given by 
$t=\sqrt{L}(\bar{\alpha}-\alpha_{0})/s_{\alpha}$, where 
\begin{equation*}
\bar{\alpha}=\frac{1}{L}\sum_{l=1}^L\widehat{\alpha}_l\quad \text{and} \quad
s_{\alpha}^{2}=\frac{1}{L-1}\sum_{l=1}^L( \widehat{\alpha}_l-\bar{\alpha}) ^{2}.
\end{equation*}
We reject the null hypothesis if $\abs{t}$ is larger than some critical value.

If $\sigma_l=\sigma$ for all $l$, the critical value is approximated by the quantile of the Student-$t$ distribution with $L-1$ degrees of freedom. If $\sigma_l$ is not homogeneous, Theorem 1 in \cite{IbragimovMueller2010} (which is due to \cite{BakirovSzekely2006}) establishes that the above $t$-test with the same critical values based on the Student-$t$ distribution is still valid but conservative. More precisely, as long as the significant level $v$ is less than $0.08$ and $L \ge 2$, we have
\begin{equation*}
\lim_{n\to \infty }P( \abs{t}
>F^{-1}_{t,L-1}(v/2) \mid H_{0}) \le v,
\end{equation*}
where $F^{-1}_{t,L-1}(v/2)$ denotes the $1-v/2$ quantile of the Student-$t$ distribution with $L-1$ degrees of freedom. Thus, an asymptotically conservative confidence interval is obtained as 
\begin{equation}
\bar{\alpha}\pm \frac{s_{\alpha }}{\sqrt{L}}F^{-1}_{t,L-1}(v/2).  \label{eq:CI}
\end{equation}
This confidence interval covers the true $\alpha$ with at least $1-v$ probability asymptotically.\footnote{\label{fn:IbragmimovMueller2010}\citet[Section 2.3]{IbragimovMueller2010} also examine the size property of the $t$-test under weak dependency and find that the test works reasonably well even if $\set{\widehat{\alpha}_1,\dots,\widehat{\alpha}_L}$ are weakly dependent. This may indicate that the $t$-statistic approach is applicable under (weak) dependence in in-sample observations as its applications only require asymptotic independence and normality of group estimators such as the tail index $\widehat{\alpha}_l$.}

Based on the previous result, we can construct the confidence intervals for $\alpha$ using the estimates from different years, which are approximately independent if the dataset for each of them are sufficiently far apart. In particular, we use the CMD estimators and the simple estimators based on \eqref{eq:alphaSpq} from every ten years using postwar top income share data. For example, we construct the confidence interval \eqref{eq:CI} for years that have identical last digit. Table \ref{t:CI} presents the (conservative) 95\% confidence intervals of the Pareto exponent using \eqref{eq:CI}. 
Because the length of confidence intervals in Table \ref{t:CI} is around 0.5, using a conservative estimate of sample size as in Figure \ref{fig:CI95} gives shorter confidence intervals. Besides, our short confidence intervals substantially refine the existing results about the Pareto exponent of income.

\begin{table}[!htb]
\centering
\caption{Conservative 95\% confidence intervals of Pareto exponents.}\label{t:CI}
\begin{tabular}{ccccc}
\toprule
& \multicolumn{2}{c}{U.S.} & \multicolumn{2}{c}{France} \\
Year & CMD 1\% & Simple 0.1/1
& CMD 1\% & Simple 0.1/1 \\ 
\cmidrule(lr){1-1}
\cmidrule(lr){2-3}
\cmidrule(lr){4-5}
XXX0 & $(1.53, 2.09)$ & $(1.64, 1.88)$  & $(1.83, 2.11)$ & $(1.74, 1.96)$ \\ 
XXX1 & $(1.61, 2.05)$ & $(1.61, 1.88)$  & $(1.80, 2.09)$ & $(1.71, 1.96)$ \\ 
XXX2 & $(1.57, 2.03)$ & $(1.62, 1.92)$  & $(1.85, 2.15)$ & $(1.77, 2.00)$ \\ 
XXX3 & $(1.61, 2.09)$ & $(1.60, 1.92)$  & $(1.84, 2.19)$ & $(1.75, 2.04)$ \\ 
XXX4 & $(1.57, 2.07)$ & $(1.56, 1.90)$  & $(1.82, 2.19)$ & $(1.75, 2.03)$ \\ 
XXX5 & $(1.52, 2.11)$ & $(1.54, 1.89)$  & $(1.80, 2.27)$ & $(1.76, 2.07)$ \\
XXX6 & $(1.50, 2.01)$ & $(1.54, 1.87)$  & $(1.83, 2.21)$ & $(1.79, 2.01)$ \\ 
XXX7 & $(1.53, 2.03)$ & $(1.54, 1.89)$  & $(1.81, 2.18)$ & $(1.78, 1.95)$ \\ 
XXX8 & $(1.51, 2.03)$ & $(1.57, 1.84)$  & $(1.84, 2.25)$ & $(1.75, 1.97)$ \\ 
XXX9 & $(1.56, 1.99)$ & $(1.60, 1.90)$  & $(1.86, 2.12)$ & $(1.77, 1.94)$ \\ 
\bottomrule
\end{tabular}
\caption*{\footnotesize Note: ``Year XXX$d$'' denotes the years with last digit $d$, \eg, $\set{1946, 1956, \dots, 2016}$ if $d=6$.  Other entries are the $95\%$ confidence intervals of the Pareto exponent using (\ref{eq:CI}). The left columns are based on the CMD estimator with top 1\% shares, and the right columns are based on the simple estimator (\ref{eq:alphaSpq}). See the main text for details of the two estimators.}
\end{table}

How should a researcher choose between the two inference approaches? We generally recommend the CMD approach based on \eqref{eq:CUMDE_lim} if the sample size $n$ is observed or at least known to be larger than some threshold, say $10^6$. Since this approach uses cross-sectional data, we can conduct inference for $\alpha$ in each year. If $n$ is completely unknown, the method by \cite{IbragimovMueller2010,IbragimovMueller2016} is the only viable alternative. Note that this method essentially requires a panel data, where we use the cross-sectional data to construct $\widehat{\alpha}_l$ in the $l$-th year and then implement the $t$-test using these estimates. In our situation, these two approaches are based on different data and hence their results are not directly comparable. In particular, the approach of \cite{IbragimovMueller2010,IbragimovMueller2016} serves as a robustness check, which leads to substantially longer confidence intervals since it relies on the estimates in only 10 years. Furthermore, this method requires the constancy of the Pareto exponent $\alpha$, which is questionable as seen in Figures \ref{fig:Pareto} and \ref{fig:ParetoFr}.

\section{Conclusion}

This paper develops an efficient minimum distance estimator of the Pareto exponent when only top income shares data are available. This is especially relevant in studying income inequality since individual level data for the top rich people are usually unavailable due to confidentiality concerns. Our estimator is consistent and asymptotically normal, and performs excellently in finite samples as shown by Monte Carlo simulations. In particular, we recommend using only top 1 instead of 10 percentile shares to study the tail of the income distribution. We estimate the Pareto exponent to be around 1.5 and stable since 1985 in U.S., and is around 1.5 and 2 before and after WWII in France.

\newpage

\appendix
\section{Proofs}\label{sec:proof}

\begin{proof}[\bf Proof of Lemma \ref{lem:StiglerPareto}]
Using \eqref{eq:muJF}, \eqref{eq:FinvPareto}, and the change of variable $v=1-x$, we obtain
\begin{align*}
\mu(J,F)&=\int_0^1 J(x)F^{-1}(x)\diff x=\int_{1-q}^{1-p} c(1-x)^{-1/\alpha}\diff x\\
&=\int_p^qcv^{-\xi}\diff v=c\frac{q^{1-\xi}-p^{1-\xi}}{1-\xi},
\end{align*}
which is \eqref{eq:mupq}. To prove \eqref{eq:sigpq}, using symmetry, \eqref{eq:sigJF}, \eqref{eq:fFinvPareto}, and the change of variable $v_i=1-x_i$, we obtain
\begin{align*}
&\sigma^2(J,F)\\
&=2\int_{0\le x_1\le x_2\le 1} \frac{J(x_1)J(x_2)}{f(F^{-1}(x_1))f(F^{-1}(x_2))}(\min\set{x_1,x_2}-x_1x_2)\diff x_1\diff x_2\\
&=2\int_{0\le x_1\le x_2\le 1} \frac{J(x_1)J(x_2)}{f(F^{-1}(x_1))f(F^{-1}(x_2))}x_1(1-x_2)\diff x_1\diff x_2\\
&=\frac{2c^2}{\alpha^2}\int_{p\le v_2\le v_1\le q}\frac{1}{v_1^{1+1/\alpha}v_2^{1+1/\alpha}}(1-v_1)v_2\diff v_1\diff v_2\\
&=2c^2\xi^2\int_p^q\int_p^{v_1}\frac{1-v_1}{v_1^{1+\xi}}v_2^{-\xi}\diff v_2\diff v_1\\
&=\frac{2c^2\xi^2}{1-\xi}\int_p^q\frac{1-v_1}{v_1^{1+\xi}}(v_1^{1-\xi}-p^{1-\xi})\diff v_1\\
&=\frac{2c^2\xi^2}{1-\xi}\int_p^q(v^{-2\xi}-v^{1-2\xi}-p^{1-\xi}v^{-1-\xi}+p^{1-\xi}v^{-\xi})\diff v\\
&=\frac{2c^2\xi^2}{1-\xi}\left(\frac{q^{1-2\xi}-p^{1-2\xi}}{1-2\xi}-\frac{q^{2-2\xi}-p^{2-2\xi}}{2-2\xi}+p^{1-\xi}\frac{q^{-\xi}-p^{-\xi}}{\xi}+p^{1-\xi}\frac{q^{1-\xi}-p^{1-\xi}}{1-\xi}\right)\\
&=\frac{2c^2\xi^2}{1-\xi}\left(\frac{q^{1-2\xi}-p^{1-2\xi}}{1-2\xi}+p^{1-\xi}\frac{q^{-\xi}-p^{-\xi}}{\xi}+\frac{2p^{1-\xi}q^{1-\xi}-p^{2-2\xi}-q^{2-2\xi}}{2-2\xi}\right),
\end{align*}
where $\frac{q^{1-2\xi}-p^{1-2\xi}}{1-2\xi}=\log \frac{q}{p}$ if $\xi=1/2$.
\end{proof}

\begin{proof}[\bf Proof of Lemma \ref{lem:Sigma}]
The formula for $\Sigma_{kk}$ follows from Lemma \ref{lem:StiglerPareto}. Suppose $j<k$ and let $v$ be the asymptotic variance of $\bar{Y}_j+\bar{Y}_k$. On the one hand, we have
$$v=\sigma_j^2+\sigma_k^2+2\Sigma_{jk}.$$
On the other hand, noting that $\bar{Y}_j+\bar{Y}_k$ is asymptotically equivalent as $L_n$ in Lemma \ref{lem:Stigler} with
$$J(x)=1[1-p_{j+1}<x\le 1-p_j]+1[1-p_{k+1}<x\le 1-p_k],$$
it follows from the proof of Lemma \ref{lem:StiglerPareto} that
$$v=I_{jj}+I_{jk}+I_{kj}+I_{kk},$$
where
\begin{equation}
I_{jk}=c^2\xi^2\int_{p_j}^{p_{j+1}}\int_{p_k}^{p_{k+1}}\frac{1}{v_1^{1+\xi}v_2^{1+\xi}}(\min\set{1-v_1,1-v_2}-(1-v_1)(1-v_2))\diff v_1\diff v_2.\label{eq:Ijk}
\end{equation}
Clearly we have $I_{jj}=\sigma_j^2$ and $I_{kk}=\sigma_k^2$. By Fubini's theorem, $I_{jk}=I_{kj}$. Therefore $\Sigma_{jk}=I_{jk}$. Since $j<k$ and hence $p_j<p_k$, we obtain
\begin{align*}
\Sigma_{jk}&=c^2\xi^2\int_{p_j}^{p_{j+1}}\int_{p_k}^{p_{k+1}}\frac{1}{v_1^{1+\xi}v_2^{1+\xi}}(1-v_1)v_2\diff v_1\diff v_2\\
&=c^2\xi^2\left(\int_{p_j}^{p_{j+1}}v_2^{-\xi}\diff v_2\right)\left(\int_{p_k}^{p_{k+1}}(v_1^{-1-\xi}-v_1^{-\xi})\diff v_1\right)\\
&=c^2\xi^2\frac{p_{j+1}^{1-\xi}-p_j^{1-\xi}}{1-\xi}\left(-\frac{p_{k+1}^{-\xi}-p_k^{-\xi}}{\xi}-\frac{p_{k+1}^{1-\xi}-p_k^{1-\xi}}{1-\xi}\right)\\
&=-c^2\xi^2\frac{p_{j+1}^{1-\xi}-p_j^{1-\xi}}{1-\xi}\left(\frac{p_{k+1}^{-\xi}-p_k^{-\xi}}{\xi}+\frac{p_{k+1}^{1-\xi}-p_k^{1-\xi}}{1-\xi}\right),
\end{align*}
which is \eqref{eq:sigpq}.

To show that $\Sigma$ is positive definite, noting that $\Sigma_{jk}=I_{jk}=I_{kj}$ and \eqref{eq:Ijk} holds, we have
$$\Sigma_{jk}=\int_{[p_1,1]^2}\phi_j(v_1)\phi_k(v_2)\frac{\min\set{1-v_1,1-v_2}-(1-v_1)(1-v_2)}{v_1v_2}\diff v_1\diff v_2,$$
where $\phi_j(v)=c\xi v^{-\xi}1[p_j\le v<p_{j+1}]$. Take any vector $z=(z_1,\dots,z_K)^\top\in \R^K$. Then as in the proof of Lemma \ref{lem:StiglerPareto}, we obtain
\begin{align*}
z^\top \Sigma z&=2\sum_{j,k}z_jz_k\int_{p_1\le v_2\le v_1\le 1}\phi_j(v_1)\phi_k(v_2)\frac{1-v_1}{v_1}\diff v_1\diff v_2\\
&=2\sum_{j,k}z_jz_k\int_{p_1}^1\int_{p_1}^{v_1}\phi_j(v_1)\phi_k(v_2)\frac{1-v_1}{v_1}\diff v_2\diff v_1\\
&=2\int_{p_1}^1\int_{p_1}^{v_1}\phi(v_1)\phi(v_2)\frac{1-v_1}{v_1}\diff v_2\diff v_1,
\end{align*}
where $\phi(v)=\sum_{k=1}^Kz_k\phi_k(v)$. Since $\phi$ is piece-wise continuous, we can take an absolutely continuous primitive function $\Phi=\int \phi$ such that $\Phi(p_1)=0$. By the fundamental theorem of calculus, we obtain
$$z^\top \Sigma z=2\int_{p_1}^1\phi(v)\Phi(v)\frac{1-v}{v}\diff v.$$
Let $I$ be the integral ignoring the factor 2. Using integration by parts, we obtain
\begin{align*}
I&=\int_{p_1}^1\phi(v)\Phi(v)\frac{1-v}{v}\diff v=\int_{p_1}^1\Phi'(v)\Phi(v)\frac{1-v}{v}\diff v\\
&=\left[\Phi(v)^2\frac{1-v}{v}\right]_{p_1}^1-\int_{p_1}^1\Phi(v)\left(\Phi'(v)\frac{1-v}{v}-\frac{\Phi(v)}{v^2}\right)\diff v\\
&=-I+\int_{p_1}^1\frac{\Phi(v)^2}{v^2}\diff v\\
\iff &z^\top \Sigma z=2I=\int_{p_1}^1\frac{\Phi(v)^2}{v^2}\diff v\ge 0,
\end{align*}
so $\Sigma$ is positive semidefinite. Since $\Phi$ is continuous, equality holds if and only if $\Phi\equiv 0\iff z=0$. Therefore $\Sigma$ is positive definite.
\end{proof}

\begin{proof}[\bf Proof of Proposition \ref{prop:normshares}]
Let $Z=(Z_1,\dots,Z_K)^\top\sim N(0,\Sigma)$. Since $\bar{Y}_k\pto \mu_k$ and $\sqrt{n}(\bar{Y}_k-\mu_k)\dto Z_k$ by \eqref{eq:multnormal}, using the definition of $S_k$, $R_k$, and $\bar{Y}_k$, we obtain
\begin{align*}
\sqrt{n}(\bar{s}_k-r_k)&=\sqrt{n}\left(\frac{S_{k+1}-S_k}{S_{K+1}-S_K}-\frac{\mu_k}{\mu_K}\right)=\sqrt{n}\left(\frac{n\bar{Y}_k/\sum Y_i}{n\bar{Y}_K/\sum  Y_i}-\frac{\mu_k}{\mu_K}\right)\\
&=\sqrt{n}\left(\frac{\bar{Y}_k}{\bar{Y}_K}-\frac{\mu_k}{\mu_K}\right)=\frac{1}{\bar{Y}_K}\sqrt{n}(\bar{Y}_k-\mu_k)-\frac{\mu_k}{\mu_K\bar{Y}_K}\sqrt{n}(\bar{Y}_K-\mu_K)\\
&\dto \frac{1}{\mu_K}Z_k-\frac{\mu_k}{\mu_K^2}Z_K.
\end{align*}
Expressing this in matrix form, we obtain
$$\sqrt{n}(\bar{s}-r)\dto HZ\sim N(0,H\Sigma H^\top).$$
Since by Lemma \ref{lem:StiglerPareto} each $\mu_k$ is proportional to $c$ and each element of $\Sigma$ is proportional to $c^2$, the vector $r$ and matrix $\Omega=H\Sigma H^\top$ depend only on $\alpha$. Since $\Sigma$ is positive definite by Lemma \ref{lem:Sigma} and $H$ has full row rank, $\Omega$ is also positive definite.
\end{proof}

\begin{proof}[\bf Proof of Proposition \ref{prop:identify}]
We prove the contrapositive. Let $\xi=1/\alpha$ and $\xi_0=1/\alpha_0$. If $r(\alpha)=r(\alpha_0)$, using $r_k=\mu_k/\mu_K$ and \eqref{eq:mupq}, in particular
\begin{align*}
r_{K-1}(\alpha)=r_{K-1}(\alpha_0)& \iff \frac{p_K^{1-\xi}-p_{K-1}^{1-\xi}}{p_{K+1}^{1-\xi}-p_K^{1-\xi}}=\frac{p_K^{1-\xi_0}-p_{K-1}^{1-\xi_0}}{p_{K+1}^{1-\xi_0}-p_K^{1-\xi_0}}\\
&\iff \frac{1-a^{1-\xi}}{b^{1-\xi}-1}=\frac{1-a^{1-\xi_0}}{b^{1-\xi_0}-1},
\end{align*}
where $a=p_{K-1}/p_K<1$ and $b=p_{K+1}/p_K>1$. By Lemma \ref{lem:monotone} below, the left-hand side is monotone in $\xi\in (0,1)$. Therefore $\xi=\xi_0$ and hence $\alpha=\alpha_0$.
\end{proof}

\begin{lem}\label{lem:monotone}
Let $a,b>0$, $a,b\neq 1$, and $a\neq b$. Then $f(x)=\frac{a^x-1}{b^x-1}$ is either strictly increasing or decreasing in $x>0$.
\end{lem}
\begin{proof}
By simple algebra, we obtain
$$f'(x)=\frac{a^x\log a(b^x-1)-b^x\log b(a^x-1)}{(b^x-1)^2}=\frac{b^x\log b}{b^x-1}\left(\frac{a^x\log a}{b^x\log b}-\frac{a^x-1}{b^x-1}\right).$$
Applying Cauchy's mean value theorem to $g_1(x)=a^x$ and $g_2(x)=b^x$, there exists $0<y<x$ such that
$$\frac{a^x-1}{b^x-1}=\frac{g_1(x)-g_1(0)}{g_2(x)-g_2(0)}=\frac{g_1'(y)}{g_2'(y)}=\frac{a^y\log a}{b^y\log b}.$$
Therefore
$$f'(x)=\frac{b^x\log b}{b^x-1}\left(\frac{a^x\log a}{b^x\log b}-\frac{a^y\log a}{b^y\log b}\right)=\frac{b^x\log a}{b^x-1}\left(\left(\frac{a}{b}\right)^x-\left(\frac{a}{b}\right)^y\right).$$
Since $a,b>0$, $a,b\neq 1$, $a\neq b$, and $0<y<x$, the sign of $\log a$ depends on $a\gtrless 1$, the sign of $b^x-1$ depends on $b\gtrless 1$, and the sign of $(a/b)^x-(a/b)^y$ depends on $a\gtrless b$. Therefore $f'(x)$ has a constant sign.
\end{proof}

\section{Simulation}\label{sec:sim}

In this appendix we evaluate the finite sample properties of the continuously updated minimum distance estimator \eqref{eq:CUMDE} through simulations.

\subsection{Simulation design}\label{subsec:sim}

We consider three data generating processes (DGPs),
\begin{inparaenum}[(i)]
\item Pareto distribution,
\item absolute value of the Student-$t$ distribution, and
\item double Pareto-lognormal distribution (dPlN).
\end{inparaenum}
For the Pareto distribution, we set the Pareto exponent to $\alpha=2$ and (without loss of generality) the minimum size to $c=1$. For the Student-$t$ distribution, we set the degree of freedom to $\nu=2$ so that the Pareto exponent is 2. The double Pareto-lognormal distribution is the product of independent double Pareto \citep{reed2001} and lognormal variables. dPlN has been documented to fit well to size distributions of economic variables including income \citep{reed2003}, city size \citep{giesen-zimmermann-suedekum2010}, and consumption \citep{Toda2017MD}. \cite{reed-jorgensen2004} show that a dPlN variable $Y$ can be generated as
$$Y=\exp(\mu+\sigma X_1+X_2/\alpha-X_3/\beta),$$
where $X_1,X_2,X_3$ are independent and $X_1\sim N(0,1)$ and $X_2,X_3\sim \mathrm{Exp}(1)$. For parameter values, we set $\mu=0$, $\sigma=0.5$, $\alpha=2$, and $\beta=1$, which are typical values for income data as documented in \cite{Toda2012JEBO}.

The simulation design is as follows. For each DGP, we generate i.i.d.\ samples with size $n=10^4,10^5,10^6$. We set the top percentiles as in \eqref{eq:TopPercent}, which are the numbers reported in \cite{piketty-saez2003}. Because the distribution is not exactly Pareto for DGP 2 (Student-$t$) and 3 (dPlN), we expect that the estimation suffers from model misspecification when we use large top income percentile such as 10\% ($p_6=0.1$). Therefore to evaluate the robustness against model misspecification, we also consider using only the top 5\% group ($p_1$--$p_5$) and the top 1\% group ($p_1$--$p_4$). Thus, in total there are $3^3=27$ specifications (three DGPs, three sample sizes, and three choices of top income percentiles). For each specification, we estimate $\widehat{\alpha}$, construct the confidence interval based on inverting the likelihood ratio test in Proposition \ref{prop:LR}, and implement the specification test in Proposition \ref{prop:spec} using the algorithm in Section \ref{subsec:implement}. The numbers are based on $M=\text{1,000}$ simulations. Table \ref{t:sim} shows the simulation results.

\begin{table}[!htb]
\centering
\caption{Finite sample properties of continuously updated minimum distance estimator.}\label{t:sim}
\begin{tabular}{crrrrrrrrr}
\toprule
DGP & \multicolumn{3}{c}{Pareto} & \multicolumn{3}{c}{$\abs{t}$} & \multicolumn{3}{c}{dPlN}\\
Top\% & 10\% & 5\% & 1\% & 10\% & 5\% & 1\% & 10\% & 5\% & 1\%\\
\cmidrule(lr){1-1}
\cmidrule(lr){2-4}
\cmidrule(lr){5-7}
\cmidrule(lr){8-10}
$n$&\multicolumn{9}{l}{Bias}\\
$10^4$&-0.02&-0.03&-0.04&-0.13&-0.07&-0.06&-0.05&-0.03&-0.04\\
$10^5$&0.00&0.00&0.00&-0.12&-0.04&-0.02&-0.04&-0.01&-0.01\\
$10^6$&0.00&0.00&0.00&-0.11&-0.04&-0.01&-0.04&0.00&0.00\\
\cmidrule(lr){1-1}
\cmidrule(lr){2-4}
\cmidrule(lr){5-7}
\cmidrule(lr){8-10}
$n$&\multicolumn{9}{l}{RMSE}\\
$10^4$&0.08&0.13&0.24&0.15&0.15&0.25&0.09&0.13&0.24\\
$10^5$&0.02&0.04&0.07&0.12&0.06&0.07&0.04&0.04&0.07\\
$10^6$&0.01&0.01&0.02&0.11&0.04&0.03&0.04&0.01&0.02\\
\cmidrule(lr){1-1}
\cmidrule(lr){2-4}
\cmidrule(lr){5-7}
\cmidrule(lr){8-10}
$n$&\multicolumn{9}{l}{Coverage}\\
$10^4$&0.92&0.92&0.92&0.50&0.86&0.90&0.85&0.91&0.90\\
$10^5$&0.96&0.94&0.95&0.00&0.76&0.94&0.59&0.93&0.95\\
$10^6$&0.92&0.95&0.95&0.00&0.04&0.91&0.00&0.92&0.96\\
\cmidrule(lr){1-1}
\cmidrule(lr){2-4}
\cmidrule(lr){5-7}
\cmidrule(lr){8-10}
$n$&\multicolumn{9}{l}{Length}\\
$10^4$&0.28&0.48&0.96&0.27&0.47&0.95&0.28&0.48&0.96\\
$10^5$&0.09&0.15&0.29&0.09&0.15&0.29&0.09&0.15&0.29\\
$10^6$&0.03&0.05&0.09&0.03&0.05&0.09&0.03&0.05&0.09\\
\cmidrule(lr){1-1}
\cmidrule(lr){2-4}
\cmidrule(lr){5-7}
\cmidrule(lr){8-10}
$n$&\multicolumn{9}{l}{Rejection probability}\\
$10^4$&0.04&0.02&0.01&0.02&0.02&0.01&0.03&0.03&0.01\\
$10^5$&0.02&0.01&0.01&0.29&0.02&0.01&0.04&0.01&0.01\\
$10^6$&0.02&0.02&0.02&1.00&0.13&0.02&0.66&0.01&0.01\\

\bottomrule
\end{tabular}
\caption*{\footnotesize Note: Each data genrating process (DGP) has Pareto exponent $\alpha=2$. $\abs{t}$: absolute value of the Student-$t$ distribution. dPlN: double Pareto-lognormal distribution. $n$: sample size. Bias: $\frac{1}{M}\sum_{m=1}^M(\widehat{\alpha}_m-\alpha)$, where $m$ indexes simulations and $M=\text{1,000}$. RMSE: root mean squared error defined by $\sqrt{\frac{1}{M}\sum_{m=1}^M(\widehat{\alpha}_m-\alpha)^2}$. ``Coverage'' is the fraction of simulations for which the true value $\alpha=2$ falls into the 95\% confidence interval. ``Length'' is the average length of confidence intervals across simulations. ``Rejection probability'' is the fraction of simulations for which the specification test in Proposition \ref{prop:spec} rejects.}
\end{table}

We can make a few observations from Table \ref{t:sim}. First, when the model is correctly specified (Pareto), the finite sample properties are excellent. In particular, the coverage rate is close to the nominal value 0.95. In this case, using more top percentiles (including the top 10\%) is more efficient (has smaller bias and RMSE) because it exploits more information. Second, when the model is misspecified (Student-$t$ or dPlN distributions), including large top percentiles (10\%) leads to large bias and incorrect coverage. Thus, it is preferable to use only percentiles within the top 1\% or 5\% for robustness against potential model misspecification. This is seen from the rejection probability of the specification test. Third, when the sample size is large ($n=10^6$, which is typical for administrative data) and we use the top 1\% group, the finite sample properties are good for all distributions considered here.

Because our estimation method is based on asymptotic normality, one may be concerned whether it is a good approximation in finite samples. To address this issue, Figure \ref{fig:sim_kernel} plots the kernel densities of $\widehat{\alpha}$ (normalized by subtracting the true value $\alpha=2$ and dividing by the sample standard deviation) based on $M=\text{1,000}$ simulations. Each figure shows the results for three sample sizes ($n=10^4, 10^5, 10^6$) as well as the standard normal density. Under the Pareto DGP, the distribution of $\widehat{\alpha}$ is very well approximated by the standard normal. Under the other two DGPs, however, when we use the top 10\% shares the distribution of $\widehat{\alpha}$ is centered far away from the true value due to model misspecification. This bias disappears as we include only small top
percentiles (\eg, only top 0--1\%), as we can see from the right panels in Figure \ref{fig:sim_kernel}.

\begin{figure}[!htbp]
\centering
\begin{subfigure}{0.3\linewidth}
\includegraphics[width=\linewidth]{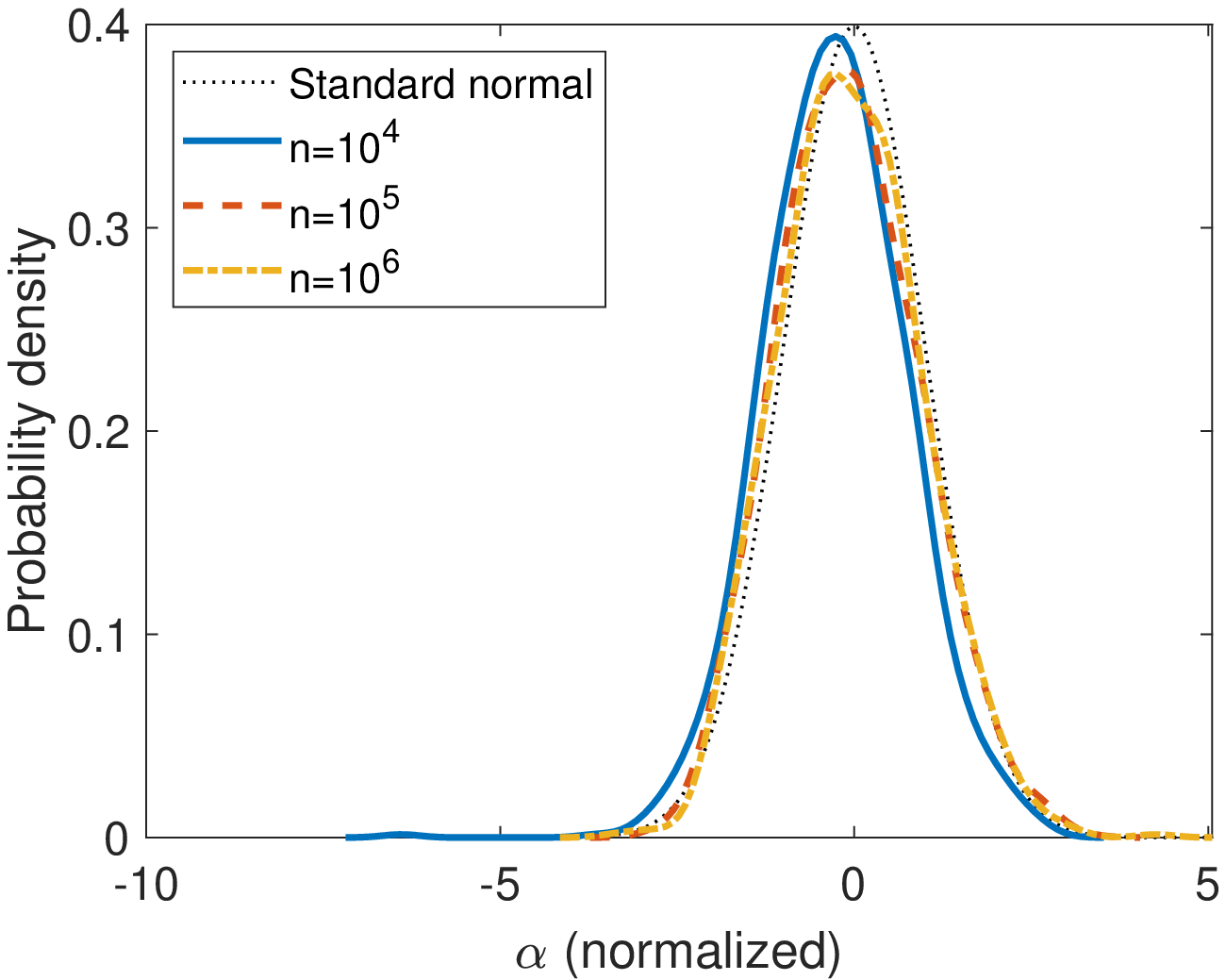}
\caption{Pareto, top 10\%}
\end{subfigure}
\begin{subfigure}{0.3\linewidth}
\includegraphics[width=\linewidth]{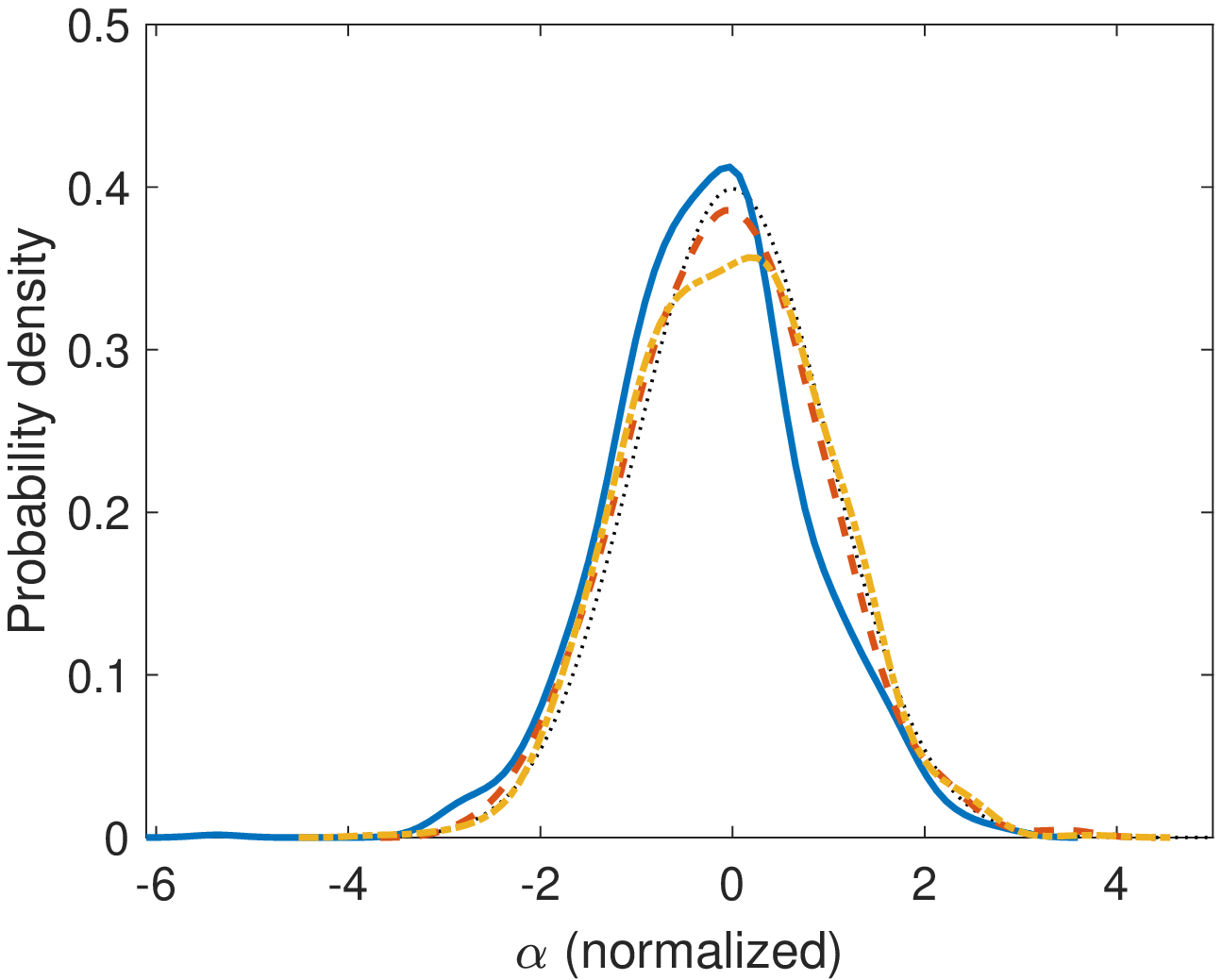}
\caption{Pareto, top 5\%}
\end{subfigure}
\begin{subfigure}{0.3\linewidth}
\includegraphics[width=\linewidth]{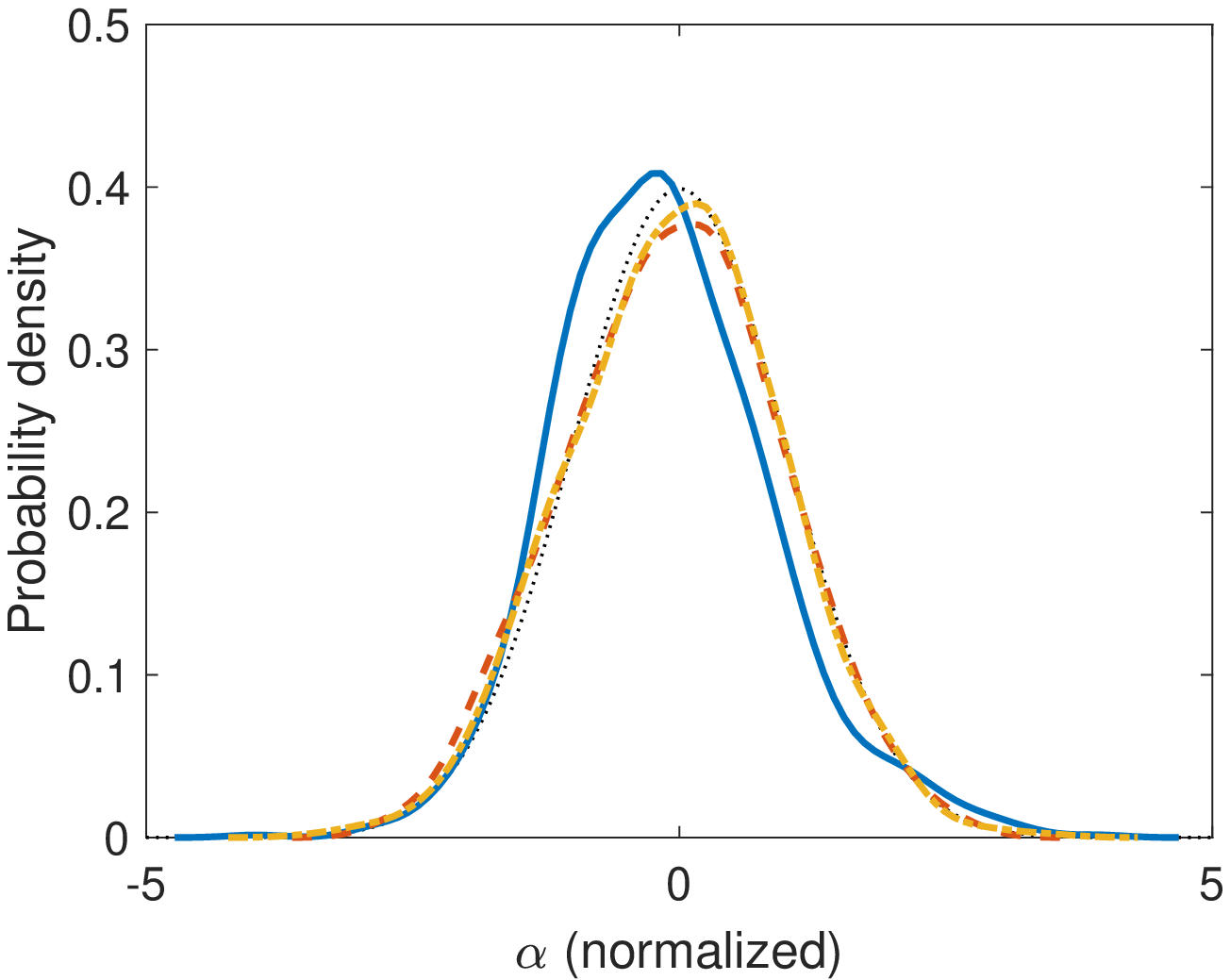}
\caption{Pareto, top 1\%}
\end{subfigure}
\begin{subfigure}{0.3\linewidth}
\includegraphics[width=\linewidth]{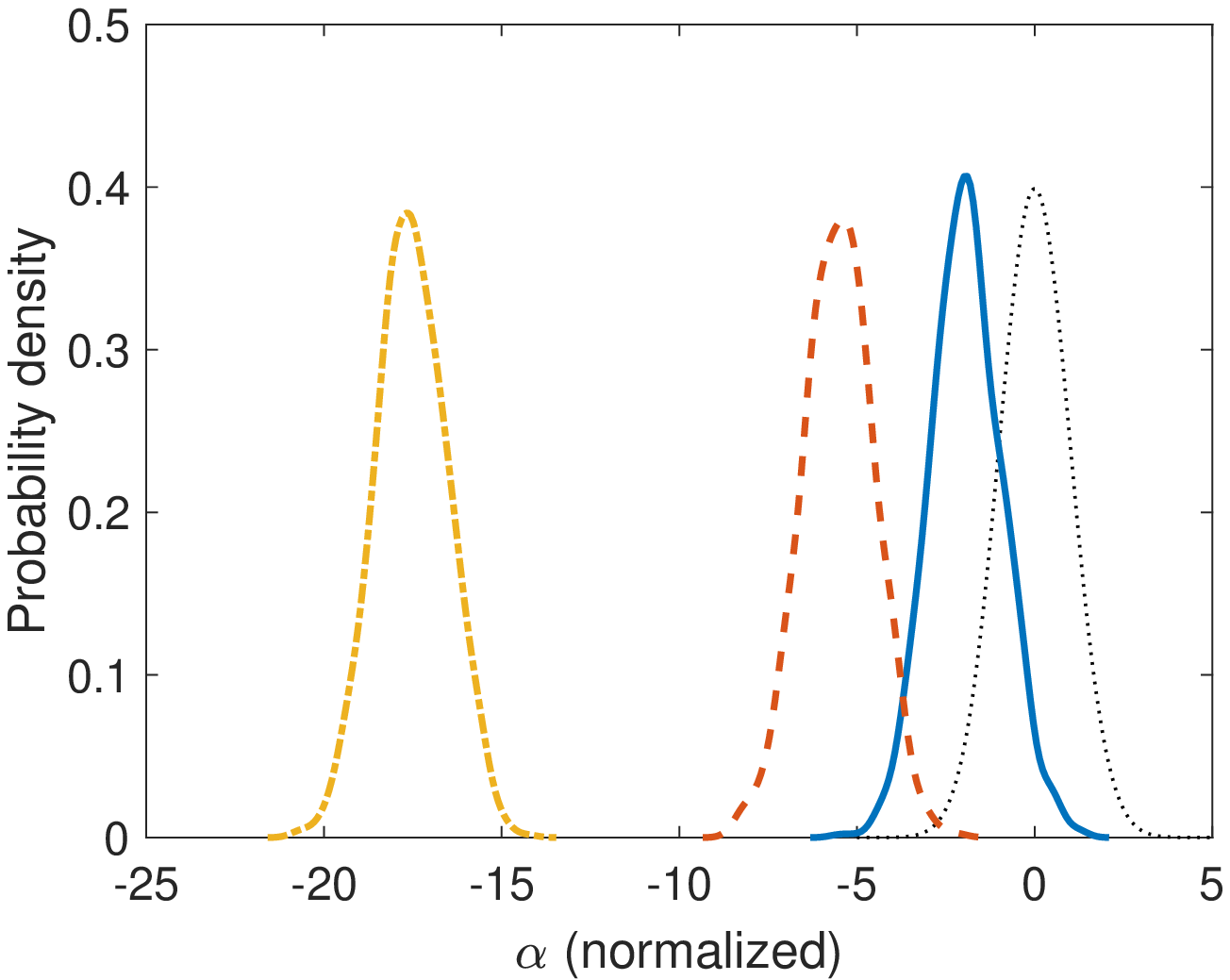}
\caption{$\abs{t}$, top 10\%}
\end{subfigure}
\begin{subfigure}{0.3\linewidth}
\includegraphics[width=\linewidth]{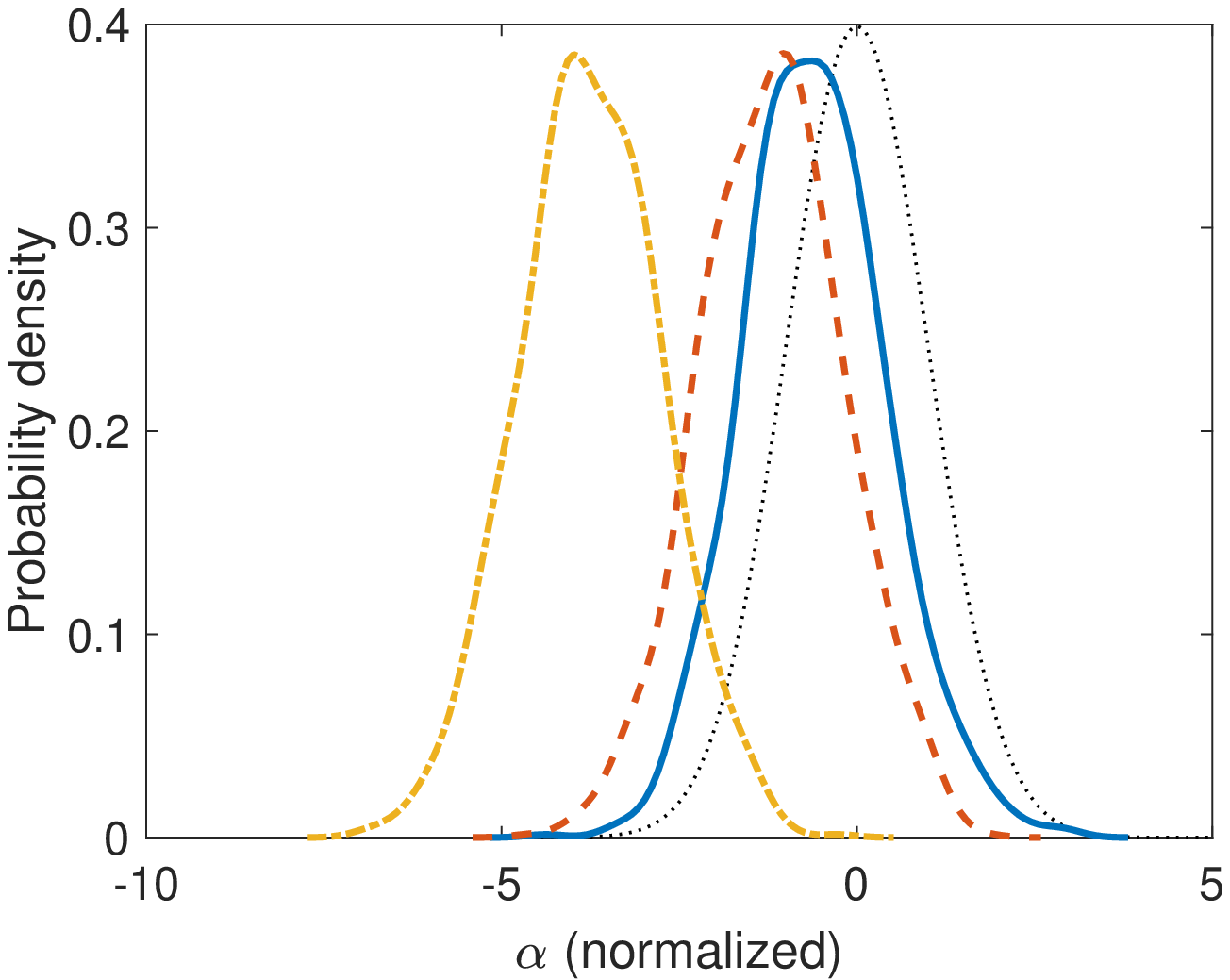}
\caption{$\abs{t}$, top 5\%}
\end{subfigure}
\begin{subfigure}{0.3\linewidth}
\includegraphics[width=\linewidth]{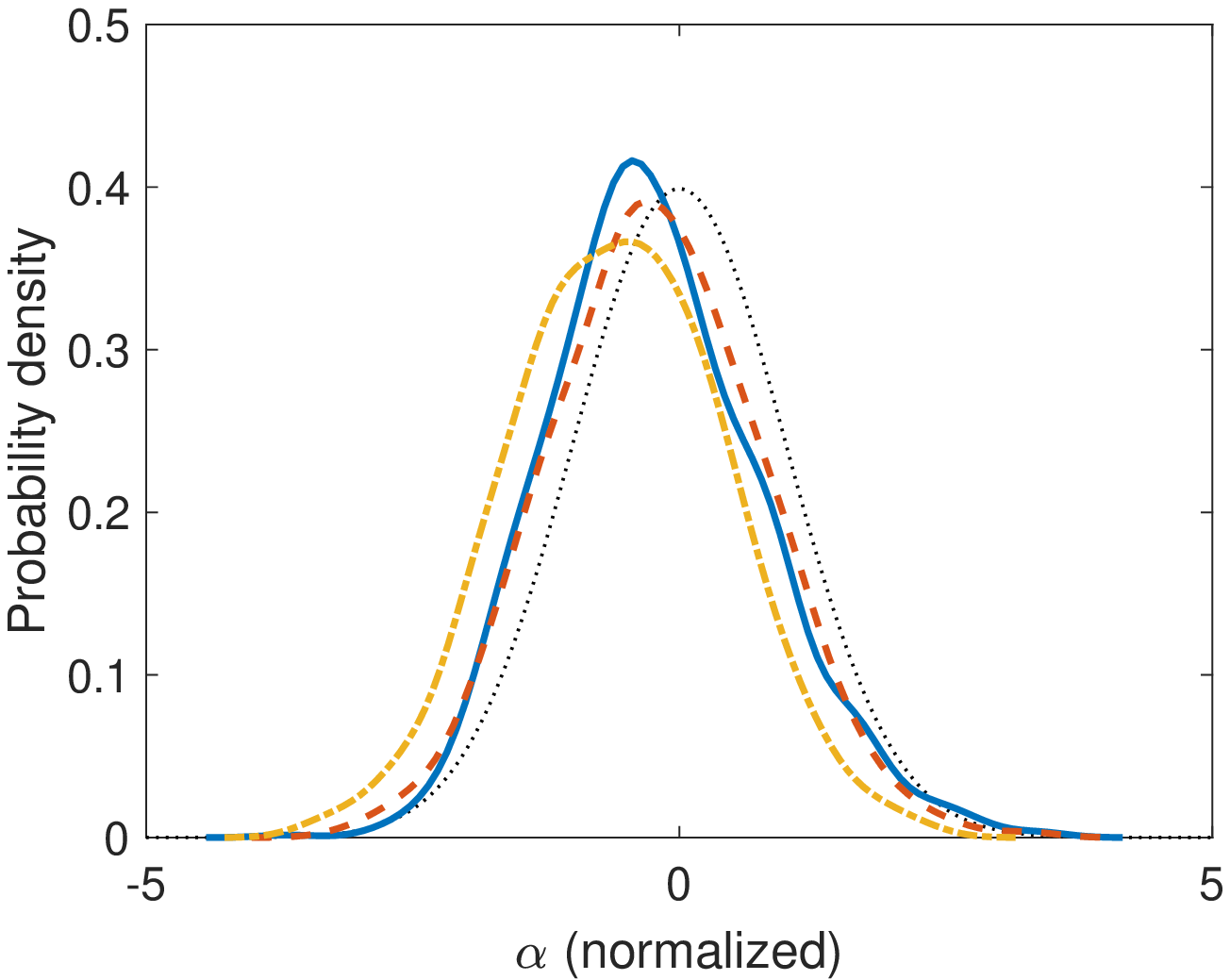}
\caption{$\abs{t}$, top 1\%}
\end{subfigure}
\begin{subfigure}{0.3\linewidth}
\includegraphics[width=\linewidth]{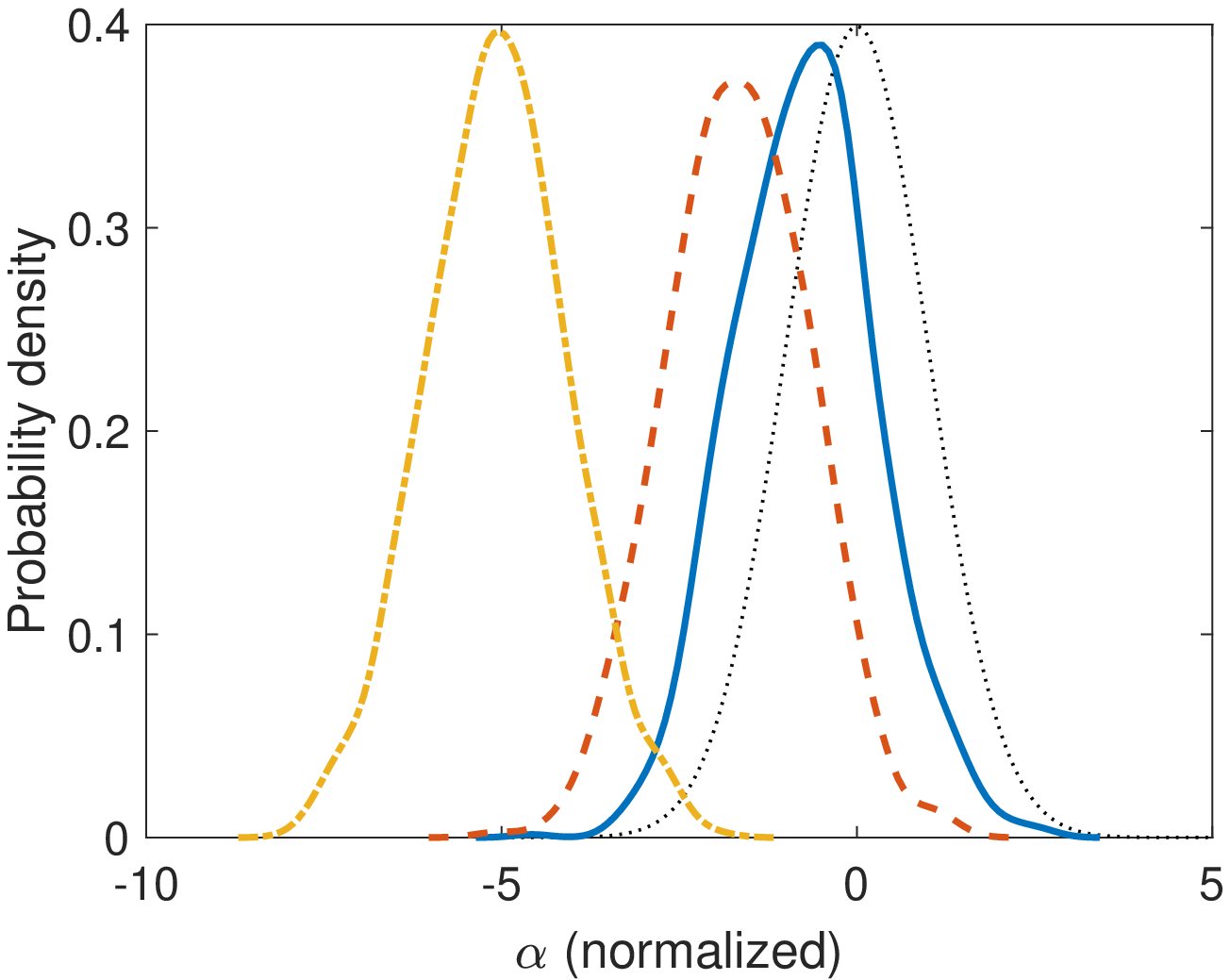}
\caption{dPlN, top 10\%}
\end{subfigure}
\begin{subfigure}{0.3\linewidth}
\includegraphics[width=\linewidth]{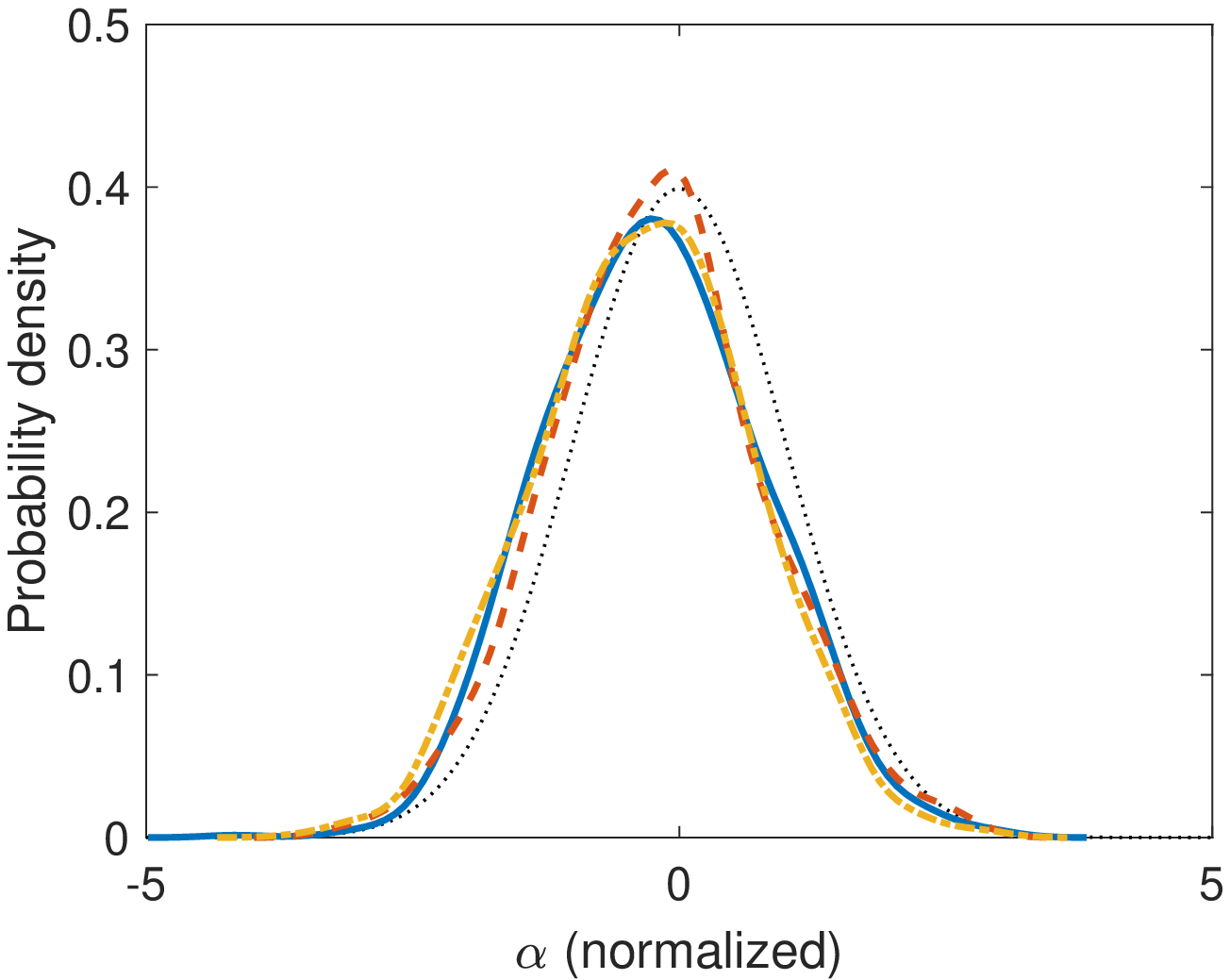}
\caption{dPlN, top 5\%}
\end{subfigure}
\begin{subfigure}{0.3\linewidth}
\includegraphics[width=\linewidth]{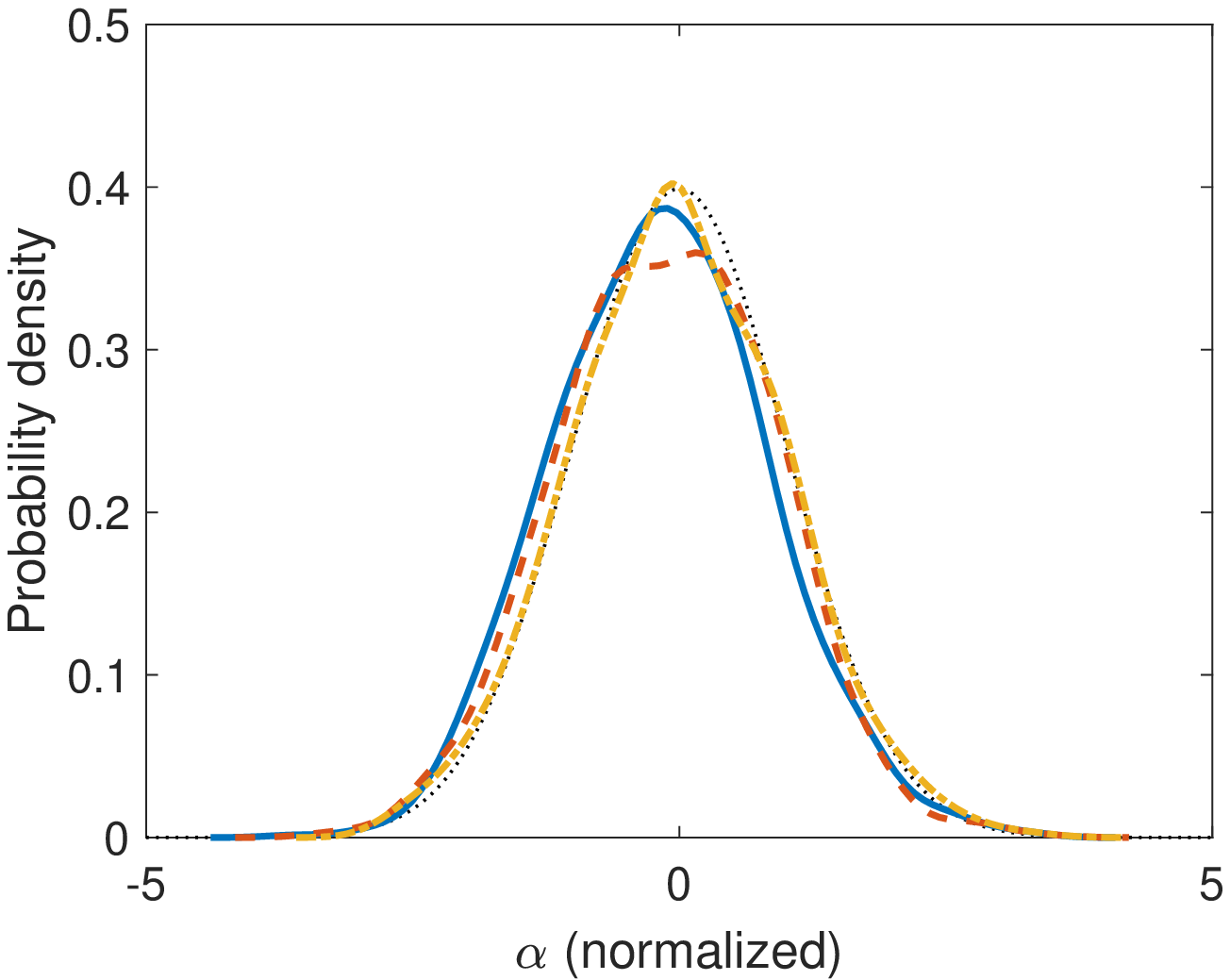}
\caption{dPlN, top 1\%}
\end{subfigure}
\caption{Kernel density of estimated Pareto exponent $\widehat{\alpha}$.}\label{fig:sim_kernel}
\caption*{\footnotesize Note: Each panel shows the kernel density of normalized $\widehat{\alpha}$ (subtracting the true value $\alpha=2$ and dividing by the sample standard deviation). See caption of Table \ref{t:sim} for the simulation design.}
\end{figure}

\subsection{Excluding small top percentiles}
Because our estimation method is based on the asymptotic distribution, one concern is that including a very small top percentile (such as $p_1=0.01\%$) may worsen the finite sample properties. To address this issue, we redo the simulation in Appendix \ref{subsec:sim} but by excluding small top percentiles. Specifically, we consider using only $p_2$--$p_6$, $p_3$--$p_6$, and $p_4$--$p_6$, where the top percentiles are given by \eqref{eq:TopPercent}. Table \ref{t:simadd} shows the results.  Compared with Table \ref{t:sim} using all percentiles ($p_1$--$p_6$, columns labeled ``10\%''), excluding the smallest top percentiles yields similar finite sample properties in terms of bias, RMSE, coverage, and length. However, the rejection probability approaches 0 as we exclude more top percentiles, so the test loses power.

\begin{table}[!htb]
\centering
\caption{Finite sample properties of continuously updated minimum distance estimator excluding small top shares.}\label{t:simadd}
\begin{tabular}{crrrrrrrrr}
	\toprule
	DGP & \multicolumn{3}{c}{Pareto} & \multicolumn{3}{c}{$\abs{t}$} & \multicolumn{3}{c}{dPIN} \\ 
	Top\% & $p_2$--$p_6$ & $p_3$--$p_6$& $p_4$--$p_6$ & $p_2$--$p_6$ & $p_3$--$p_6$ & $p_4$--$p_6$ & $p_2$--$p_6$ &  $p_3$--$p_6$ & $p_4$--$p_6$ \\
	\cmidrule(lr){1-1}
	\cmidrule(lr){2-4}
	\cmidrule(lr){5-7}
	\cmidrule(lr){8-10}
	$n$ & \multicolumn{9}{l}{Bias} \\ 
	$10^4$ & -0.01 & -0.00 & 0.00 & -0.12 & -0.12 & -0.12 & -0.04 & -0.04 & -0.04 \\ 
	$10^5$ & 0.00 & 0.00 & 0.00 & -0.12 & -0.12 & -0.13 & -0.04 & -0.04 & -0.04 \\ 
	$10^6$ & 0.00 & 0.00 & 0.00 & -0.11 & -0.12 & -0.13 & -0.04 & -0.04 & -0.04 \\
	\cmidrule(lr){1-1}
	\cmidrule(lr){2-4}
	\cmidrule(lr){5-7}
	\cmidrule(lr){8-10}
	$n$ & \multicolumn{9}{l}{RMSE} \\ 
	$10^4$ & 0.07 & 0.07 & 0.08 & 0.14 & 0.14 & 0.14 & 0.08 & 0.08 & 0.08 \\ 
	$10^5$ & 0.02 & 0.02 & 0.02 & 0.12 & 0.12 & 0.13 & 0.04 & 0.04 & 0.05 \\ 
	$10^6$ & 0.01 & 0.01 & 0.01 & 0.12 & 0.12 & 0.13 & 0.04 & 0.04 & 0.04 \\
	\cmidrule(lr){1-1}
	\cmidrule(lr){2-4}
	\cmidrule(lr){5-7}
	\cmidrule(lr){8-10}
	$n$ & \multicolumn{9}{l}{Coverage} \\ 
	$10^4$ & 0.94 & 0.95 & 0.96 & 0.55 & 0.58 & 0.57 & 0.90 & 0.92 & 0.91 \\ 
	$10^5$ & 0.95 & 0.95 & 0.95 & 0.00 & 0.00 & 0.00 & 0.61 & 0.63 & 0.58 \\ 
	$10^6$ & 0.92 & 0.92 & 0.92 & 0.00 & 0.00 & 0.00 & 0.00 & 0.00 & 0.00 \\
	\cmidrule(lr){1-1}
	\cmidrule(lr){2-4}
	\cmidrule(lr){5-7}
	\cmidrule(lr){8-10}
	$n$ & \multicolumn{9}{l}{Length} \\ 
	$10^4$ & 0.29 & 0.30 & 0.31 & 0.28 & 0.28 & 0.29 & 0.28 & 0.29 & 0.30 \\ 
	$10^5$ & 0.09 & 0.09 & 0.10 & 0.09 & 0.09 & 0.09 & 0.09 & 0.09 & 0.09 \\ 
	$10^6$ & 0.03 & 0.03 & 0.03 & 0.03 & 0.03 & 0.03 & 0.03 & 0.03 & 0.03 \\
	\cmidrule(lr){1-1}
	\cmidrule(lr){2-4}
	\cmidrule(lr){5-7}
	\cmidrule(lr){8-10}
	$n$ & \multicolumn{9}{l}{Rejection probability} \\ 
	$10^4$ & 0.01 & 0.01 & 0.00 & 0.02 & 0.02 & 0.00 & 0.01 & 0.01 & 0.00 \\ 
	$10^5$ & 0.02 & 0.02 & 0.00 & 0.34 & 0.33 & 0.00 & 0.06 & 0.06 & 0.00 \\ 
	$10^6$ & 0.02 & 0.02 & 0.00 & 1.00 & 1.00 & 0.00 & 0.67 & 0.67 & 0.00 \\
	\bottomrule
\end{tabular}
\caption*{\footnotesize Note: See caption of Table \ref{t:sim} for the simulation design.}
\end{table}

\subsection{Comparison with the simple estimator}\label{subsec:simple}

In this appendix we compare the finite sample performance of our classical minimum distance estimator (CMD) of the Pareto exponent to the simple estimator in \eqref{eq:alphaSpq}. For the simple estimator, we set $(p,q)=(0.1,1)/100$ as is common in the literature, and we also consider $(p,q)=(0.1,0.5)/100, (0.5,1)/100$. For the CMD estimator, to make the results comparable, we use $(p_1,p_2,p_3,p_4)=(0.01,0.1,0.5,1)/100$. Table \ref{t:simple} shows the results. According to the table, the CMD estimator uniformly outperforms the simple estimator in \eqref{eq:alphaSpq} in terms of bias and RMSE.

\begin{table}[!htb]
\centering
\caption{Finite sample properties of continuously updated minimum distance estimator and the simple estimator in \eqref{eq:alphaSpq}.}\label{t:simple}
\begin{tabular}{crrrrrrrr}
	\toprule
	& \multicolumn{4}{c}{Bias} & \multicolumn{4}{c}{RMSE} \\ 
	$100(p,q)$ & CMD & (.1,1) & (.1,.5) & (.5,1) & CMD & (.1,1) & (.1,.5) & (.5,1) \\
	\cmidrule(lr){1-1}
	\cmidrule(lr){2-5}
	\cmidrule(lr){6-9}
	$n$ & \multicolumn{8}{l}{Pareto} \\
	$10^4$ & -0.04 &0.19 &0.24 &0.09& 0.24 &0.44 &0.53 &0.29\\
	$10^5$ & 0.00 &0.03 &0.04 &0.02&0.07 &0.15 &0.18 &0.11\\
	$10^6$ & 0.00 &0.00 &0.00 &0.00&0.02 &0.06 &0.07 &0.04\\
	\cmidrule(lr){1-1}
	\cmidrule(lr){2-5}
	\cmidrule(lr){6-9}
	$n$ & \multicolumn{8}{l}{$\abs{t}$} \\
	$10^4$ & -0.06 &0.17 &0.23 &0.07 &0.25 &0.42 &0.52 &0.28\\
	$10^5$ & -0.02 &0.03 &0.04 &0.01 &0.07 &0.16 &0.18 &0.11\\
	$10^6$ & -0.01 &0.00 &0.00 &-0.01 &0.03 &0.05 &0.06 &0.04\\
	\cmidrule(lr){1-1}
	\cmidrule(lr){2-5}
	\cmidrule(lr){6-9}
	$n$ & \multicolumn{8}{l}{dPlN} \\
	$10^4$ & -0.04& 0.19 &0.25 &0.10 &0.24 &0.43 &0.53 &0.27\\
	$10^5$ & -0.01& 0.03 &0.03 &0.01 &0.07 &0.14 &0.17 &0.10\\
	$10^6$ & 0.00 &0.00 &0.01 &0.00 &0.02 &0.06 &0.07 &0.05\\
	\bottomrule
\end{tabular}
\caption*{\footnotesize Note: See caption of Table \ref{t:sim} for the simulation design. ``CMD'' refers to the continuously updated minimum distance estimator with top 0.01, 0.1, 0.5, 1 percentiles. $100(p,q)$ denotes the top percentiles used in the simple estimator \eqref{eq:alphaSpq}.}
\end{table}


\end{document}